\documentclass[english]{amsart}
\usepackage{a4wide}

\newtheorem{theorem}{Theorem}[section]
\newtheorem{lemma}[theorem]{Lemma}

\newtheorem{proposition}[theorem]{Proposition}

\theoremstyle{definition}
\newtheorem{definition}[theorem]{Definition}

\numberwithin{equation}{section}

\theoremstyle{remark}

\usepackage{graphicx}

\usepackage[fixlanguage]{babelbib}
\selectbiblanguage{english}

\usepackage[shortlabels]{enumitem}

\usepackage{amsfonts,amssymb,amsbsy,amsmath,mathrsfs,amsthm}

\usepackage{commath}

\usepackage{hyperref}
\usepackage{url}

\hypersetup{pdfpagemode=UseNone, pdfstartview=FitH,
  colorlinks=true,urlcolor=red,linkcolor=blue,citecolor=black,
  pdftitle={Default Title, Modify},pdfauthor={Your Name},
  pdfkeywords={Modify keywords}}

\RequirePackage{babel}

\providecommand{\norm}[1]{ \lVert#1  \rVert}
\newcommand{\dmu}{\, d \mu}
\newcommand{\dx}{\, d x}

\newcommand{\dt}{\, d t}

\newcommand{\dla}{\, d \lambda}

\def\Xint#1{\mathchoice
   {\XXint\displaystyle\textstyle{#1}}%
   {\XXint\textstyle\scriptstyle{#1}}%
   {\XXint\scriptstyle\scriptscriptstyle{#1}}%
   {\XXint\scriptscriptstyle\scriptscriptstyle{#1}}%
   \!\int}
\def\XXint#1#2#3{{\setbox0=\hbox{$#1{#2#3}{\int}$}
     \vcenter{\hbox{$#2#3$}}\kern-.5\wd0}}

\def\dashint{\Xint-}

\DeclareMathOperator*{\esssup}{ess\,sup}

\usepackage{cite}
\makeatletter
\newcommand{\citecomment}[2][]{\citen{#2}#1\citevar}
\newcommand{\citeone}[1]{\citecomment{#1}}
\newcommand{\citetwo}[2][]{\citecomment[,~#1]{#2}}
\newcommand{\citevar}{\@ifnextchar\bgroup{;~\citeone}{\@ifnextchar[{;~\citetwo}{]}}}
\newcommand{\citefirst}{\@ifnextchar\bgroup{\citeone}{\@ifnextchar[{\citetwo}{]}}}

\makeatother

\begin{document}

\title{Characterizations of parabolic reverse H\"older classes}

\author{Juha Kinnunen}
\address{Department of Mathematics, Aalto University, P.O. Box 11100, FI-00076 Aalto, Finland}
\email{juha.k.kinnunen@aalto.fi}

\author{Kim Myyryl\"ainen}
\address{Department of Mathematics, Aalto University, P.O. Box 11100, FI-00076 Aalto, Finland}
\email{kim.myyrylainen@aalto.fi}
\thanks{The second author was supported by the Magnus Ehrnrooth Foundation.}



\begin{abstract}
This paper discusses parabolic reverse H\"older inequalities and their connections to parabolic Muckenhoupt weights. The main result gives several characterizations for this class of weights. There are challenging features related to the parabolic geometry and the time lag, for example, in covering and chaining arguments. We also prove a Gehring type self-improving property for parabolic reverse H\"older inequalities.

\end{abstract}

\maketitle

\section{Introduction}

This paper continues and complements a discussion of parabolic reverse H\"older inequalities and Muckenhoupt weights in~\cite{KinnunenMyyry2023} and \cite{kinnunenSaariMuckenhoupt,kinnunenSaariParabolicWeighted}.
We attempt to create a higher dimensional version of the one-dimensional theory introduced by Sawyer~\cite{sawyer1986} and studied, for example, by
Cruz-Uribe, Neugebauer and  Olesen~\cite{CUNO1995}, Mart\'\i n-Reyes, Pick and de la Torre~\cite{MRPT1993},  Mart\'\i n-Reyes and de la Torre~\cite{MRT1994}. 
Our approach is motivated by certain doubly nonlinear parabolic partial differential equations as in~\cite{KinnunenMyyry2023,kinnunenSaariMuckenhoupt,kinnunenSaariParabolicWeighted}.
Several challenges occur compared to the standard theory of weighted norm inequalities.
For example, the doubling property of Muckenhoupt weights is replaced by a forward in time doubling property in~\cite{KinnunenMyyry2023,kinnunenSaariParabolicWeighted}. 
A parabolic Muckenhoupt weight satisfies a forward in time doubling property, but it is not currently known whether the same holds true for a weight satisfying a parabolic reverse H\"older inequality.
There are also interesting features related to the parabolic geometry and the time lag.
In contrast with the parabolic Muckenhoupt classes, a parabolic reverse H\"older inequality with a positive time lag implies the corresponding condition with zero time lag.
Alternative higher dimensional versions have been studied by Berkovits~\cite{berkovits2011}, Forzani, Mart\'{\i}n-Reyes and Ombrosi~\cite{ForzaniMartinreyesOmbrosi2011}, Lerner and Ombrosi~\cite{LO2010} and Ombrosi~\cite{Ombrosi2005}. 
However, the geometries in these approaches are different from ours.

Let $1<p<\infty$, $x\in\mathbb R^n$, $L>0$ and $t \in \mathbb{R}$.
A parabolic rectangle centered at $(x,t)$ with side length $L$ is
\[
R = R(x,t,L) = Q(x,L) \times (t-L^p, t+L^p)
\]
and its upper and lower parts are
\[
R^+(\gamma) = Q(x,L) \times (t+\gamma L^p, t+L^p)
\]
and
\[
R^-(\gamma) = Q(x,L) \times (t - L^p, t - \gamma L^p) ,
\]
where $0 \leq \gamma < 1$ is called the time lag. 
Here $Q(x,L)=\{y \in \mathbb R^n: \lvert y_i-x_i\rvert \leq \frac L2,\,i=1,\dots,n\}$ denotes a spatial cube
with center $x$ and side length $L$. 

Let $1<q<\infty$.
A nonnegative weight $w$ belongs to the parabolic reverse H\"older class $RH^+_q$ if
there exists a constant $C$ such that
\[
\biggl( \dashint_{R^-(\gamma)} w^{q} \biggr)^\frac{1}{q} \leq C \dashint_{R^+(\gamma)} w
\]
for every parabolic rectangle $R \subset \mathbb R^{n+1}$.
Lemma \ref{lem:RHItimelag} shows that the definition of $RH^+_q$ does not depend on the time lag.
In other words, if a weight belongs to $RH^+_q$ with some time lag, it belongs to $RH^+_q$ with any time lag.
Reverse H\"older inequalities are closely related to Muckenhoupt weights.
A weight $w$ satisfies a parabolic Muckenhoupt condition, if 
\begin{equation*}
\sup_{R \subset \mathbb R^{n+1}}\biggl( \dashint_{R^-(\gamma)} w \biggr) \biggl( \dashint_{R^+(\gamma)} w^{\frac{1}{1-q}} \biggr)^{q-1} < \infty.
\end{equation*}
Parabolic Muckenhoupt classes are independent of the positive time lag $\gamma>0$, see \cite[Proposition 3.4 (vii)]{kinnunenSaariParabolicWeighted} and \cite[Theorem 3.1]{KinnunenMyyry2023}.
Every parabolic Muckenhoupt weight satisfies a parabolic reverse H\"older inequality, see \cite[Theorem 5.2]{kinnunenSaariParabolicWeighted} and \cite[Theorem 5.2]{KinnunenMyyry2023}.
Conversely, Theorem \ref{rhar} shows that a weight satisfying the  parabolic reverse H\"older inequality is a parabolic Muckenhoupt weight under the assumption that the weight satisfies a forward in time parabolic doubling condition in~\eqref{eq:pardoubling}.

Our main result Theorem \ref{thm:RHIchar} gives several characterizations of the parabolic reverse H\"older inequality.
We also study the corresponding limiting class $RH^+_\infty$ in Proposition \ref{rhinfty}.
Self-improving phenomena are essential in the theory of Muckenhoupt weights and reverse H\"older inequalities.
Theorem \ref{gehring} is a parabolic Gehring type higher integrability result, which asserts that
\[
w\in RH^+_q
\Longrightarrow
w\in RH^+_{q+\varepsilon}
\]
for some $\varepsilon>0$. 
The characterizations of parabolic reverse H\"older inequalities and the parabolic Gehring lemma also hold in the case $p=1$ which extends the corresponding one-dimensional results.

\section{Definition and properties of parabolic reverse H\"older inequalities}

Throughout the underlying space is $\mathbb{R}^{n+1}=\{(x,t):x=(x_1,\dots,x_n)\in\mathbb R^n,t\in\mathbb R\}$.
Unless otherwise stated, constants are positive and the dependencies on parameters are indicated in the brackets.
The Lebesgue measure of a subset $A$ of $\mathbb{R}^{n+1}$ is denoted by $\lvert A\rvert$.
The integral average of $f \in L^1(A)$ in measurable set $A\subset\mathbb{R}^{n+1}$, with $0<|A|<\infty$, is denoted by
\[
f_A = \dashint_A f \dx \dt = \frac{1}{\lvert A\rvert} \int_A f(x,t)\dx\dt .
\]

Instead of Euclidean cubes, we work with the collection of parabolic rectangles $R = R(x,t,L)$ in $\mathbb{R}^{n+1}$.
The spatial side length of a parabolic rectangle $R$ is denoted by $l_x(R)=L$ and the time length by $l_t(R)=2L^p$.
We write $R^\pm$ for $R^{\pm}(0)$ in the case with zero time lag.
The top of a rectangle $R = R(x,t,L)$ is $Q(x,L) \times\{t+L^p\}$
and the bottom is $Q(x,L) \times\{t-L^p\}$.
The $\lambda$-dilate of $R$  with $\lambda>0$ is denoted by $\lambda R = R(x,t,\lambda L)$.

This section discusses basic properties of parabolic reverse H\"older inequalities. 
We begin with the definition of the uncentered parabolic maximal functions.
The differentials $\dx \dt$ in integrals are omitted in the sequel.

\begin{definition}
Let $f$ be a locally integrable function. 
The uncentered forward in time and backward in time parabolic maximal functions are defined by
\[
M^{+}f(x,t) = \sup_{R^-\ni(x,t)} \dashint_{R^+} \lvert f \rvert
\]
and
\[
M^{-}f(x,t) = \sup_{R^+\ni(x,t)} \dashint_{R^-} \lvert f \rvert.
\]
\end{definition}

A locally integrable nonnegative function $w$ is called a weight.
We begin with the definitions of parabolic reverse H\"older classes $RH^+_q$ and $RH^+_\infty$.
It is enough to consider the case with zero time lag, since Lemma \ref{lem:RHItimelag} below shows that the time lag does not play any role in the definitions.

\begin{definition}
Let $1<q<\infty$.
A weight $w$ belongs to the parabolic reverse H\"older class $RH^+_q$ if
there exists a constant $C=[w]_{RH^+_q}$ such that
\[
\biggl( \dashint_{R^-} w^{q} \biggr)^\frac{1}{q} \leq C \dashint_{R^+} w
\]
for every parabolic rectangle $R \subset \mathbb R^{n+1}$.
If the condition above holds with the time axis reversed, then $w \in RH^-_q$.
\end{definition}

\begin{definition}
A weight $w$ belongs to the parabolic reverse H\"older class $RH^+_\infty$ if
there exists a constant $C=[w]_{RH^+_\infty}$ such that
\[
\esssup_{R^-} w \leq C \dashint_{R^+} w
\]
for every parabolic rectangle $R \subset \mathbb R^{n+1}$.
If the condition above holds with the time axis reversed, then $w \in RH^-_\infty$.
\end{definition}

We discuss characterizations  for $RH^+_\infty$.
Compare Proposition~\ref{rhinfty}~$(ii)$ with Theorem~\ref{thm:RHIchar}~$(ii)$
and Proposition~\ref{rhinfty}~$(iii)$ with Theorem~\ref{thm:RHIchar}~$(vi)$ below.

\begin{proposition}
\label{rhinfty}
Let $w$ be a weight.
The following conditions are equivalent.
\begin{enumerate}[(i),topsep=5pt,itemsep=5pt]
\item $w \in RH^+_\infty$.

\item There exists a constant $C$ such that
\[
\frac{w(E)}{w(R^+)} 
\leq
C \frac{\lvert E \rvert}{\lvert R^- \rvert} 
\]
for every measurable set $E\subset R^-$.

\item There exists a constant $C$ such that
\[
M^+(w \chi_{R^-})(x,t) \leq C w_{R^+}
\]
for every $(x,t)\in R^-$.

\end{enumerate}

\end{proposition}

\begin{proof}
First we show that $(i)\Leftrightarrow (ii)$.
Assume that $(i)$ holds and let $E\subset R^-$ be a measurable set. Then
\[
w(E) = \int_{R^-} w \chi_E \leq \lvert E \rvert \esssup_{R^-} w 
\leq C w_{R^+} \lvert E \rvert.
\]
This proves $(ii)$. 
Then assume that $(ii)$ holds. Let $E_\lambda = R^- \cap \{w>\lambda\}$, $\lambda>0$. 
We have
\[
\lambda \lvert E_\lambda \rvert \leq w(E_\lambda) \leq C w_{R^+} \lvert E_\lambda \rvert ,
\]
which implies that $\lambda \leq C w_{R^+}$ when $\lvert E_\lambda \rvert >0$.
Thus, 
we obtain $(i)$ since
\[
\esssup_{R^-} w = \sup \{ \lambda: \lvert E_\lambda \rvert >0 \} \leq C w_{R^+} .
\]

Then we show that $(i)\Leftrightarrow (iii)$.
We observe that $(i)$ implies $(iii)$ since
\[
M^+(w \chi_{R^-})(x,t) = \sup_{P^-\ni (x,t)} \dashint_{P^+} w \chi_{R^-} \leq \esssup_{R^-} w \leq C w_{R^+}
\]
for every $(x,t) \in R^-$.
Then we show that $(iii)$ implies $(i)$.
By the Lebesgue differentiation theorem~\cite[Lemma~2.3]{KinnunenMyyryYang2022} and $(iii)$,
we have
\[
w(x,t) 
\leq M^+(w \chi_{R^-})(x,t) \leq C w_{R^+}
\]
for almost every $(x,t) \in R^-$.
By taking the essential supremum over $R^-$, we obtain $(i)$.
\end{proof}

Next we show that the parabolic reverse H\"older classes do not depend on the time lag.

\begin{lemma}
\label{lem:RHItimelag}
Let $1<q\leq\infty$ and $0<\gamma<1$. Then $w$ belongs to $RH_q^+$ if and only if there exists a constant $C$ such that
\[
\biggl( \dashint_{R^-(\gamma)} w^{q} \biggr)^\frac{1}{q} \leq C \dashint_{R^+(\gamma)} w
\]
for every parabolic rectangle $R \subset \mathbb R^{n+1}$.
\end{lemma}

\begin{proof}
Assume that $w\in RH_q^+$.
Let $R \subset \mathbb R^{n+1}$ be a parabolic rectangle with side length $L$.
Choose $N\in\mathbb{N}$ and $0<\beta\leq1$ such that
$1+\gamma = (N + \beta)(1-\gamma)$.
Let
\[
R_0^+(\gamma) = R^-(\gamma) + (0, \beta (1-\gamma) L^p)
\]
and
\[
R_k^+(\gamma) = R^-(\gamma) + (0, (k+\beta) (1-\gamma) L^p)
\] 
for $k=1,\dots,N$.
Note that $R_N^+(\gamma) = R^+(\gamma)$.
Let $\rho = \beta^{1/p} (1-\gamma)^{1/p}$.
We partition $R^-(\gamma)$ into $ \lceil \rho^{-1} \rceil^n \lceil \rho^{-p} \rceil$ subrectangles $S^-_{i}$ with spatial side length $\rho L$ and time length $\rho^p L^p$ such that the overlap of $\{S^-_{i}\}_i$ is bounded by $2^{n+1}$. 
This can be done by dividing each spatial edge of $R^-(\gamma)$ into $\lceil \rho^{-1} \rceil$ equally long 
subintervals with an overlap bounded by $2$, and the time interval of $R^-(\gamma)$ into $\lceil \rho^{-p} \rceil$ equally long subintervals with an overlap bounded by $2$.
We observe that every $S^{+}_i$ is contained in $R^+_0(\gamma)$.
Then $w\in RH_q^+$
implies that there exists a constant $C_1$ such that
\begin{align*}
\biggl( \dashint_{R^-(\gamma)} w^{q} \biggr)^\frac{1}{q} 
&\leq 
\biggl( \sum_i \frac{\lvert S^-_i \rvert}{\lvert R^-(\gamma) \rvert} \dashint_{S^-_i} w^{q} \biggr)^\frac{1}{q} 
\leq
\biggl( \frac{\rho^{n+p}}{1-\gamma} \biggr)^\frac{1}{q}
\sum_i \biggl( \dashint_{S^-_i} w^{q} \biggr)^\frac{1}{q} \\
&\leq 
\bigl( \beta^{\frac{n}{p}+1} (1-\gamma)^\frac{n}{p} \bigr)^\frac{1}{q} 
C_1 \sum_i \dashint_{S^{+}_i} w \\
&= 
\bigl( \beta^{\frac{n}{p}+1} (1-\gamma)^\frac{n}{p} \bigr)^\frac{1}{q} 
C_1 \sum_i \frac{\lvert R^+_0(\gamma) \rvert}{\lvert S^{+}_i \rvert} \frac{1}{\lvert R^+_0(\gamma) \rvert} \int_{S^{+}_i} w \\
&\leq
\bigl( \beta^{\frac{n}{p}+1} (1-\gamma)^\frac{n}{p} \bigr)^{\frac{1}{q}-1}
C_1
2^{n+1} \dashint_{R^{+}_0(\gamma)} w 
= 
C_2
\dashint_{R^{+}_0(\gamma)} w ,
\end{align*}
where $C_2 = \bigl( \beta^{\frac{n}{p}+1} (1-\gamma)^\frac{n}{p} \bigr)^{\frac{1}{q}-1} C_1 2^{n+1}$.

By iterating the previous argument with $1$ in place of $\beta$, we obtain
\begin{align*}
\dashint_{R^{+}_0(\gamma)} w &\leq \biggl( \dashint_{R^{+}_0(\gamma)} w^{q} \biggr)^\frac{1}{q} 
\leq 
\frac{C_1 2^{n+1}}{(1-\gamma)^{\frac{n }{p} (1-\frac{1}{q}) }}
\dashint_{R^{+}_1(\gamma)} w \\
&\leq C_3^N \dashint_{R^{+}_N(\gamma)} w
\leq C_4 \dashint_{R^{+}(\gamma)} w ,
\end{align*}
where $C_3 = (1-\gamma)^{\frac{n }{p} 
(\frac{1}{q}-1)
} C_1 2^{n+1}$
and
$C_4 = \max\{1, C_3^\frac{1+\gamma}{1-\gamma} \}$.
Thus, we conclude that
\[
\biggl( \dashint_{R^-(\gamma)} w^{q} \biggr)^\frac{1}{q} 
\leq 
C_2
\dashint_{R^{+}_0(\gamma)} w
\leq
C_2
C_4 \dashint_{R^{+}(\gamma)} w .
\]
By letting $q\to\infty$, we obtain the same conclusion for $RH^+_\infty$.

Then we prove the other direction.
Let $R \subset \mathbb R^{n+1}$ be a parabolic rectangle with side length $L$.
We partition $R^-$ into $ 2^n \lceil (1+\gamma)/(1-\gamma) \rceil$ subrectangles $S^-_{i}(\gamma)$ with spatial side length $L/(1+\gamma)^\frac{1}{p}$ and time length $(1-\gamma) L^p/(1+\gamma)$ such that the overlap of $\{S^-_{i}(\gamma)\}_i$ is bounded by $2^{n+1}$. 
This can be done by dividing each spatial edge of $R^-$ into 
$\lceil (1+\gamma)^\frac{1}{p} \rceil = 2$  equally long subintervals,
and the time interval of $R^-$ into $\lceil (1+\gamma)/(1-\gamma) \rceil$ equally long subintervals with an overlap bounded by $2$.
We observe that every $S^{+}_i(\gamma)$ is contained in $R^+$.
Then by the assumption, we have
\begin{align*}
\biggl( \dashint_{R^-} w^{q} \biggr)^\frac{1}{q} 
&\leq 
\biggl( \sum_i \frac{\lvert S^-_i(\gamma) \rvert}{\lvert R^- \rvert} \dashint_{S^-_i(\gamma)} w^{q} \biggr)^\frac{1}{q} 
\leq
\biggl( \frac{1-\gamma}{(1+\gamma)^{\frac{n}{p}+1}} \biggr)^\frac{1}{q}
\sum_i \biggl( \dashint_{S^-_i(\gamma)} w^{q} \biggr)^\frac{1}{q} \\
&\leq 
C_1^\frac{1}{q}
C \sum_i \dashint_{S^{+}_i(\gamma)} w 
= 
C_1^\frac{1}{q}
C \sum_i \frac{\lvert R^+ \rvert}{\lvert S^{+}_i(\gamma) \rvert} \frac{1}{\lvert R^+ \rvert} \int_{S^{+}_i(\gamma)} w \\
&\leq C_1^{\frac{1}{q}-1}
C2^{n+1} \dashint_{R^{+}} w ,
\end{align*}
where $C_1 = (1-\gamma)/ (1+\gamma)^{\frac{n}{p}+1} $.
This completes the proof for $1<q<\infty$. Letting $q\to\infty$ in the argument above, we obtain the claim for $q=\infty$.
\end{proof}

\section{Characterizations of $\bigcup_{q>1} RH^+_q$}

This section discusses several characterizations of parabolic reverse H\"older inequalities in terms of conditions that resemble characterizations of the Muckenhoupt $A_\infty$ class 
in the classical setting.
Reverse H\"older classes and Muckenhoupt classes require separate discussion in the parabolic case. The connection between these classes is discussed in Section \ref{section:muckenhoupt}.
The results in this section also hold in the case $p=1$.

\begin{theorem}
\label{thm:RHIchar}
Let $w$ be a weight.
The following conditions are equivalent.
\begin{enumerate}[(i),topsep=5pt,itemsep=5pt]
\item $w\in RH^+_q$ for some $1<q<\infty$.

\item There exist constants $C,\delta >0$ such that
\[
\frac{w(E)}{w(R^+)} 
\leq
C \biggl( \frac{\lvert E \rvert}{\lvert R^- \rvert} \biggr)^\delta 
\]
for every parabolic rectangle $R\subset\mathbb{R}^{n+1}$ and 
measurable set $E \subset R^-$.

\item For every $\beta>0$ there exists $0<\alpha<1$ such that for every parabolic rectangle $R\subset\mathbb{R}^{n+1}$ and every measurable set $E \subset R^-$ for which $\lvert E \rvert < \alpha \lvert R^- \rvert$ we have $w(E) < \beta w(R^+)$.

\item There exist $0<\alpha<1$ and $0<\beta<1/2^{n+p}$ such that for every parabolic rectangle $R\subset\mathbb{R}^{n+1}$ and every measurable set $E \subset R^-$ for which $\lvert E \rvert < \alpha \lvert R^- \rvert$ we have $w(E) < \beta w(R^+)$.

\item There exist $0<\alpha<1$ and $0<\beta<1/2^{n+p}$ such that for every parabolic rectangle $R\subset\mathbb{R}^{n+1}$ we have
\[
w( R^- \cap \{ \alpha w > w_{R^+} \} ) < \beta w( R^+ ) .
\]

\item 
There exists a constant $C$ such that
\[
\int_{R^-} M^+(w \chi_{R^-}) 
\leq 
C \int_{R^+} w
\]
for every parabolic rectangle $R\subset\mathbb{R}^{n+1}$.

\item
There exists a constant $C$ such that
\[
\int_{R^-} w \log^+ \biggl(\frac{w}{w_{R+}}\biggr) \leq C w(R^+)
\]
for every parabolic rectangle $R\subset\mathbb{R}^{n+1}$.
\end{enumerate}
\end{theorem}

The proof is presented in the subsections below.

\subsection{Quantitative measure condition}

We show $(i)\Leftrightarrow (ii)$ in Theorem~\ref{thm:RHIchar}.

\begin{theorem}
\label{thm:quantiAinfty}
Let $w$ be a weight. 
Then $w\in RH^+_q$ for some $1<q<\infty$ if and only if there exist constants $C,\delta >0$ such that
\[
\frac{w(E)}{w(R^+)} 
\leq
C \biggl( \frac{\lvert E \rvert}{\lvert R^- \rvert} \biggr)^\delta 
\]
for every parabolic rectangle $R\subset\mathbb{R}^{n+1}$ and 
measurable set $E \subset R^-$.

\end{theorem}

\begin{proof}
Assume first that $w\in RH^+_q$.
Let $E$ be a measurable subset of $R^-$.
By H\"older's inequality, we have
\begin{align*}
\frac{w(E)}{w(R^+)} &= \frac{\lvert E \vert}{w(R^+)} \dashint_{E} w \leq \frac{\lvert E \vert}{w(R^+)} \biggl( \dashint_{E} w^q \biggr)^\frac{1}{q} \\
&\leq \frac{\lvert E \vert^{1-\frac{1}{q} }}{w(R^+)} \lvert R^- \rvert^\frac{1}{q} \biggl( \dashint_{R^-} w^q \biggr)^\frac{1}{q} 
\leq \frac{\lvert E \vert^{1-\frac{1}{q} }}{w(R^+)} \lvert R^- \rvert^\frac{1}{q} C \dashint_{R^+} w\\
&= C \lvert E \vert^{1-\frac{1}{q}} \lvert R^- \rvert^{\frac{1}{q}-1} 
\leq C \biggl( \frac{\lvert E \vert}{\lvert R^- \vert} \biggr)^{1-\frac{1}{q}} .
\end{align*}

Then we prove the other direction.
Assume that
\[
\frac{w(E)}{w(R^+)} 
\leq
C \biggl( \frac{\lvert E \rvert}{\lvert R^- \rvert} \biggr)^\frac{1}{q} ,
\]
where $C >0$, $q=\delta^{-1}>0$ and $E$ is a measurable subset of $R^-$.
Since the ratio of the Lebesgue measure of $R^-$ to the Lebesgue measure of $E$ is always greater than or equal to 1, we may assume without loss of generality that the exponent $q$ is strictly greater than 1.
Let $E_\lambda = R^- \cap \{ w > \lambda \}$.
We have $\lvert E_\lambda \rvert \leq w(E_\lambda) / \lambda$.
It follows that
\[
\lvert E_\lambda \rvert \leq \frac{1}{\lambda} w(E_\lambda)
\leq
\frac{C}{\lambda} \biggl( \frac{\lvert E_\lambda \rvert}{\lvert R^- \rvert} \biggr)^\frac{1}{q} w(R^+) ,
\]
and hence we get
\[
\lvert E_\lambda \rvert \leq \frac{C^{q'}}{\lambda^{q'}} \frac{w(R^+)^{q'}}{\lvert R^- \rvert^{q'-1}} ,
\]
where $q'=\frac{q}{q-1}$ is the conjugate exponent of $q$.
Letting $0<\varepsilon<q'-1$ and applying Cavalieri's principle gives
\begin{align*}
\int_{R^-} w^{1+\varepsilon} &= (1+\varepsilon) \int_0^\infty \lambda^{\varepsilon} \lvert R^- \cap \{ w > \lambda \} \rvert \dla \\
&=
(1+\varepsilon) \int_0^{w_{R^+}} \lambda^{\varepsilon} \lvert E_\lambda \rvert \dla + (1+\varepsilon) \int_{ w_{R^+} }^\infty \lambda^{\varepsilon} \lvert E_\lambda \rvert \dla \\
&\leq
\lvert R^- \rvert \biggl( \frac{w(R^+)}{\lvert R^+ \rvert} \biggr)^{1+\varepsilon} + (1+\varepsilon) C^{q'} \frac{w(R^+)^{q'}}{\lvert R^- \rvert^{q'-1}} \int_{w_{R^+}}^\infty \lambda^{\varepsilon-q'} \dla \\
&=
\lvert R^- \vert \biggl( \frac{w(R^+)}{\lvert R^+ \rvert} \biggr)^{1+\varepsilon} + 
\frac{(1+\varepsilon) C^{q'}}{q'-1-\varepsilon} \frac{w(R^+)^{q'}}{\lvert R^- \rvert^{q'-1}} \biggl( \frac{w(R^+)}{\lvert R^+ \rvert} \biggr)^{\varepsilon-q'+1} \\
&= \biggl( 1 + \frac{(1+\varepsilon) C^{q'}}{q'-1-\varepsilon} \biggr) \lvert R^- \vert \biggl( \frac{w(R^+)}{\lvert R^+ \rvert} \biggr)^{1+\varepsilon} .
\end{align*}
Thus, we obtain
\[
\biggl( \int_{R^-} w^{1+\varepsilon} \biggr)^\frac{1}{1+\varepsilon} \leq c \int_{R^+} w ,
\]
where $c^{1+\varepsilon} = 1 + (1+\varepsilon) C^{q'} / (q'-1-\varepsilon)$.
By taking the supremum over all parabolic rectangles, 
we conclude that $w\in RH^+_{1+\varepsilon}$ 
and thus the proof is complete.
\end{proof}

\subsection{Qualitative measure condition}

We show $(i)\Leftrightarrow (iv)$ in Theorem~\ref{thm:RHIchar}.
First we note
that Theorem~\ref{thm:RHIchar}~$(ii)$ implies $(iii)$,
since if $\lvert E \rvert < \alpha \lvert R^- \rvert$, then
\begin{align*}
w(E) \leq C \biggl( \frac{\lvert E \rvert}{\lvert R^- \rvert} \biggr)^\delta w(R^+) \leq C \alpha^\delta w(R^+) ,
\end{align*}
where we can choose $\alpha$ small enough such that $C\alpha^\delta\leq\beta$.
The implication from $(iii)$ to $(iv)$ is immediate.

To prove the reverse implication from $(iv)$ to $(i)$,
we need the following lemma.
We present the version with a time lag for later use.

\begin{lemma}
\label{lemma:timemove}
Let $0\leq\gamma<1$.
Assume that there exist $0<\alpha,\beta<1$ such that for every parabolic rectangle $R$ and every measurable set $E \subset R^-(\gamma)$ for which $\lvert E \rvert < \alpha \lvert R^-(\gamma) \rvert$ we have $w(E) < \beta w(R^+(\gamma))$.
Then we have the following properties.
\begin{enumerate}[(i)]
\item 
For every parabolic rectangle $R$ and every measurable set $E \subset R^-(\gamma)$ for which $w(E) \geq \beta w(R^+(\gamma))$ it holds that $\lvert E \rvert \geq \alpha \lvert R^-(\gamma) \rvert$.
\item
Let $\theta>0$.
For every parabolic rectangle $R$ and $0\leq\eta\leq \theta $
it holds that
\[
w(R^-(\gamma)) \leq C w(R^-(\gamma) + (0, \eta L^p)),
\]
where $C\geq1$ depends on $p, \gamma, \alpha,\beta$ and $\theta $.
\end{enumerate}
\end{lemma}

\begin{proof}

$(i)$ This is simply the contraposition of the qualitative measure condition.

$(ii)$
We first prove the claim for $\eta=1$. Partition $R^-(\gamma)$ into $\lceil \alpha^{-1} \rceil$ pairwise disjoint sets $E_i$ with measure at most $\alpha \lvert R^-(\gamma) \rvert$.
Then
the qualitative measure condition implies $w(E_i) < \beta w(R^+(\gamma))$,
and thus
\begin{equation}
\begin{split}
\label{eq:lowerupperpart_bound}
w(R^-(\gamma)) 
&= \sum_i w(E_i) \leq \sum_i \beta w(R^+(\gamma))\\
&= \lceil \alpha^{-1} \rceil \beta w(R^+(\gamma))
\leq C_0 w(R^+(\gamma)) ,
\end{split}
\end{equation}
where $C_0 = \max\{1, 2\beta/\alpha \}$.
This finishes the proof in the case $\eta=1$.

It is enough to prove the claim for $\eta=\theta$.
The general case $0\leq\eta\leq \theta $ follows from the fact that the constant $C$ in the claim is an increasing function of $\theta $.
Let $\theta>0$ and $R \subset \mathbb{R}^{n+1}$ be a fixed parabolic rectangle of side length $L$.
Choose $m \in \mathbb{N}$ such that
\begin{equation}
\label{partitionbound}
\frac{(1+\gamma) L^p}{2^{pm}} \leq 
\frac{(1-\gamma) L^p}{2}
< \frac{(1+\gamma) L^p}{2^{p(m-1)}} .
\end{equation}
We partition $R^-(\gamma)$ into subrectangles $R^-_{0,i}(\gamma)$ with spatial side length $L /2^m$ and time length $(1-\gamma) L^p /2^{pm}$ such that the overlap of $\{R^-_{0,i}(\gamma)\}_i$ is bounded by $2$. This can be done by dividing each spatial edge of $R^-(\gamma)$ into $2^m$ equally long pairwise disjoint intervals, and the time interval of $R^-(\gamma)$ into $\lceil 2^{pm} \rceil$ equally long subintervals such that their overlap is bounded by $2$.

Our plan is to shift every rectangle $R^-_{0,i}(\gamma)$ forward in time by multiple times of $(1+\gamma) L^p / 2^{pm}$ until the shifted rectangles are contained in $R^-(\gamma) + (0, \theta L^p )$.
To this end, choose
$N \in \mathbb{N}$ such that
\[
(N-1) \frac{(1+\gamma) L^p}{2^{pm}}  < \theta L^p \leq N \frac{(1+\gamma) L^p}{2^{pm}} .
\]
We first move every rectangle $R^-_{0,i}(\gamma)$ forward in time by $(N-1) (1+\gamma) L^p / 2^{pm}$. Then we shift once more by the distance $(1+\gamma) L^p / 2^{pm}$ those rectangles that are not yet subsets of $R^-(\gamma) + (0, \theta L^p )$.
Denote so obtained shifted rectangles by $R^-_{N,i}(\gamma)$. 
Observe that the choice of $N$ 
and \eqref{partitionbound}
ensures that all shifted rectangles $R^-_{N,i}(\gamma)$ are contained in $R^-(\gamma) + (0, \theta L^p) $.
By the construction and the bounded overlap of $R^-_{0,i}(\gamma)$, the overlap of $R^-_{N,i}(\gamma)$ is bounded by $4$.
Then we apply~\eqref{eq:lowerupperpart_bound} for $R^-_{0,i}(\gamma)$ and $R^+_{0,i}(\gamma)$
and continue applying~\eqref{eq:lowerupperpart_bound} for shifted rectangles total of $N$ times to obtain
\[
w(R^-_{0,i}(\gamma)) \leq C_0 w(R^+_{0,i}(\gamma)) \leq C_0^N w(R^-_{N,i}(\gamma)) ,
\]
where
\[
C_0^{N} \leq C_0^{1+ 2^{pm} \theta / (1+\gamma) } \leq C_0^{1 + 2^{p+1} \theta /(1-\gamma)} = C .
\]
Therefore, we conclude that
\begin{align*}
w(R^-(\gamma)) 
&\leq \sum_i w(R^-_{0,i}(\gamma)) \leq 
C\sum_i w(R^-_{N,i}(\gamma)) \\
&\leq 4 
Cw(R^-(\gamma) + (0, \theta L^p))
\end{align*}
by $R^-_{N,i}(\gamma)\subset R^-(\gamma) + (0, \theta L^p)$ and the bounded overlap of $R^-_{N,i}(\gamma)$.
\end{proof}

\begin{lemma}
\label{weightmeasure-estimate}
Let $w$ be a weight.
Assume that there exist 
$0<\alpha<1$ and $0<\beta<1/2^{n+p}$
such that for every parabolic rectangle $R$ and every measurable set $E \subset R^-$ for which $\lvert E \rvert < \alpha \lvert R^- \rvert$ we have $w(E) < \beta w(R^+)$.
Then there exists $c=c(p,\alpha,\beta)$ such that for every parabolic rectangle $R=R(x,t,L) \subset \mathbb{R}^{n+1}$ and $\lambda \geq w_{U^+}$ we have
\[
w( R^{-} \cap \{ w > \lambda \}  )  \leq c \lambda \lvert R \cap \{ w > (1-2^{n+p}\beta) \lambda \} \rvert ,
\]
where
$U^+ = R^+ + (0, \tau L^p)$ with 
$\tau = 1/(2^p-1)$. 
\end{lemma}

\begin{proof}

Let $R_0=R(x_0,t_0,L) = Q(x_0,L) \times (t_0-L^p, t_0+L^p)$
and $\lambda \geq w_{U^+_0}$.
Without loss of generality, we may assume that $\alpha < 1/2^{n+p}$.
Denote $S^-_0 = R^-_0$. The time length of $S^-_0$ is $l_t(S^-_0) = L^p$.
We construct a parabolic Calder\'on--Zygmund decomposition at level $\lambda$.
We partition $S^-_0$ by dividing each spatial edge into $2$ equally long intervals. If
\[
\frac{l_t(S_{0}^-)}{\lfloor 2^{p} \rfloor} < \frac{L^p}{2^{p}},
\]
we divide the time interval of $S^-_0$ into $\lfloor 2^{p} \rfloor$ equally long intervals. Otherwise, we divide the time interval of $S^-_0$ into $\lceil 2^{p} \rceil$ equally long intervals.
We obtain subrectangles $S^-_1$ of $S^-_0$ with spatial side length $L_1 = l_x(S^-_1) = l_x(S^-_0)/2 = L / 2$
and time length either
\[
l_t(S^-_1) = \frac{l_t(S^-_0)}{\lfloor 2^{p} \rfloor} = \frac{L^p}{\lfloor 2^{p} \rfloor} \quad \text{or} \quad l_t(S^-_1) = \frac{L^p}{\lceil 2^{p} \rceil} .
\]
For every $S^-_1$, there exists a unique rectangle $R^-_1$ with spatial side length $L_1 = L / 2$ 
and time length $L_1^p = L^p / 2^{p}$
such that $R^-_1$ has the same bottom as $S^-_1$,
unless the top of $S^-_1$ intersects with the top of $S^-_0$ in which case we choose $R^-_1$ that has the same top as $S^-_1$.
This way every $R_1^-$ is contained in $S^-_0$ and their overlap is bounded by $3$.
Consider the corresponding 
$U^+_1 = R^-_1 + (0, (1+\tau) L_1^p)$.
We select those rectangles $S^-_1$ for which
\[
\frac{w(U^+_1)}{\lvert U^+_1 \rvert} = \dashint_{U^+_1} w > \lambda
\]
and denote the obtained collection by $\{ S^-_{1,j} \}_j$.
If
\[
\frac{w(U^+_1)}{\lvert U^+_1 \rvert} = \dashint_{U^+_1} w \leq \lambda ,
\]
we subdivide $S^-_1$ in the same manner as above
and select all those subrectangles $S^-_2$ for which
\[
\frac{w(U^+_2)}{\lvert U^+_2 \rvert} = \dashint_{U^+_2} w > \lambda
\]
to obtain family $\{ S^-_{2,j} \}_j$.
We continue this selection process recursively.
At the $i$th step, we partition unselected rectangles $S^-_{i-1}$ by dividing each spatial side into $2$ equally long intervals. If
\begin{equation}
\label{RHI:JNproof_eq1}
\frac{l_t(S_{i-1}^-)}{\lfloor 2^{p} \rfloor} < \frac{L^p}{2^{pi}},
\end{equation}
we divide the time interval of $S^-_{i-1}$ into $\lfloor 2^{p} \rfloor$ equally long intervals.  
Otherwise, if
\begin{equation}
\label{RHI:JNproof_eq2}
\frac{l_t(S_{i-1}^-)}{\lfloor 2^{p} \rfloor} \geq \frac{L^p}{2^{pi}},
\end{equation}
we divide the time interval of $S^-_{i-1}$ into $\lceil 2^{p} \rceil$ equally long intervals.
We obtain subrectangles $S^-_i$. 
For every $S^-_i$, there exists a unique rectangle $R^-_i$ with spatial side length $L_i = L / 2^{i}$
and time length $L_i^p = L^p / 2^{pi}$
such that $R^-_i$ has the same bottom as $S^-_i$,
unless the top of $S^-_i$ intersects with the top of $S^-_{i-1}$ in which case we choose $R^-_i$ that has the same top as $S^-_i$.
This way every $R_i^-$ is contained in $S^-_{i-1}$ and their overlap is bounded by $3$.
Consider the corresponding 
$U^+_i = R^-_i + (0, (1+\tau) L_i^p)$.
Select those $S^-_i$ for which 
\begin{equation}
\label{eq:selectcrit1}
\frac{w(U^+_i)}{\lvert U^+_i \rvert} = \dashint_{U^+_i} w  > \lambda
\end{equation}
and denote the obtained collection by $\{ S^-_{i,j} \}_j$.
If 
\[
\frac{w(U^+_i)}{\lvert U^+_i \rvert} = \dashint_{U^+_i} w \leq \lambda ,
\]
we continue the selection process in $S^-_i$.
In this manner we obtain a collection $\{S^-_{i,j} \}_{i,j}$ of pairwise disjoint rectangles.

We show that
\begin{equation}\label{eq:tsidel}
\frac{1}{2} \frac{L^p}{2^{pi}} \leq l_t(S^-_i) \leq \frac{L^p}{2^{pi}}
\end{equation}
for every $S^-_i$.
Observe that if \eqref{RHI:JNproof_eq1} holds, then we have
\[
l_t(S_i^-) = \frac{l_t(S^-_{i-1})}{\lfloor 2^{p} \rfloor} < \frac{L^p}{2^{pi}}.
\]
On the other hand, if \eqref{RHI:JNproof_eq2} holds, then
\[
l_t(S_i^-) = \frac{l_t(S^-_{i-1})}{\lceil 2^{p} \rceil} \leq \frac{l_t(S^-_{i-1})}{2^{p}} \leq \dots \leq \frac{L^p}{2^{pi}} .
\]
This gives an upper bound in \eqref{eq:tsidel}.

Suppose that \eqref{RHI:JNproof_eq2} is satisfied at the $i$th step.
Then we have a lower bound for the time length of $S_i^-$, since
\[
l_t(S^-_i) = \frac{l_t(S_{i-1}^-)}{\lceil 2^{p} \rceil} \geq \frac{\lfloor 2^{p} \rfloor}{\lceil 2^{p} \rceil} \frac{L^p}{2^{pi}} \geq \frac{1}{2} \frac{L^p}{2^{pi}} .
\]
On the other hand, if \eqref{RHI:JNproof_eq1} is satisfied, then
\[
l_t(S^-_i) = \frac{l_t(S_{i-1}^-)}{\lfloor 2^{p} \rfloor} \geq \frac{l_t(S_{i-1}^-)}{ 2^{p}}.
\]
In this case, \eqref{RHI:JNproof_eq2} has been satisfied at an earlier step $i'$ with $i'< i$.
We obtain
\[
l_t(S^-_i) \geq \frac{l_t(S_{i-1}^-)}{ 2^{p}} \geq \dots \geq \frac{l_t(S_{i'}^-)}{ 2^{p(i-i')}} \geq \frac{1}{2} \frac{L^p}{ 2^{pi}}
\]
by using the lower bound for $S_{i'}^-$.
This proves \eqref{eq:tsidel}.

We show that $U^+_i$ is contained in $U^-_{i-1} = R^-_{i-1} + (0, \tau L_{i-1}^p)$ 
for a fixed rectangle $S^-_{i-1}$ and for every subrectangle $S^-_i \subset S^-_{i-1}$, where $S^-_{i-1}$ is the parent rectangle of $S^-_i$.
Since  $\tau = 1/(2^p-1)$ and $L_i = L / 2^{i}$, we have $(1+\tau) L_i^p= \tau L_{i-1}^p$.
By using the fact that $R^-_i \subset S^-_{i-1} \subset R^-_{i-1}$, we obtain
\begin{align*}
U^+_i &= R^-_i + (0, (1+\tau) L_i^p)
\subset R^-_{i-1} + (0, \tau L_{i-1}^p)
=
U^-_{i-1} .
\end{align*}

We have a collection $\{ S^-_{i,j} \}_{i,j}$ of pairwise disjoint rectangles. 
However, the rectangles in the corresponding collection $\{ U^+_{i,j} \}_{i,j}$ may overlap. 
Thus, we replace it by a maximal subfamily $\{ \widetilde{U}^+_{i,j} \}_{i,j}$ of pairwise disjoint rectangles, which is constructed in the following way.
For every $i\in\mathbb{N}$,
we may extract a maximal disjoint subcollection $\{ \widehat{U}^+_{i,j} \}_{j}$ from 
$\{ U^+_{i,j} \}_{j}$
such that 
for every $U^+_{i,j}$ there is $\widehat{U}^+_{i,j}$ with
\[
\text{pr}_x(U^+_{i,j}) \subset \text{pr}_x(\widehat{U}^+_{i,j}) 
\quad\text{and}\quad  
\text{pr}_t(U^+_{i,j}) \subset 3 \text{pr}_t(\widehat{U}^+_{i,j}) .
\]
Here pr$_x$ denotes the projection to $\mathbb R^n$ and pr$_t$ denotes the projection to the time axis.
Choose $\{ \widehat{U}^+_{1,j} \}_{j}$ and denote it by $\{ \widetilde{U}^+_{1,j} \}_j$.
Then consider the collection $\{ \widehat{U}^+_{2,j} \}_{j}$ where each $\widehat{U}^+_{2,j}$ either intersects some $\widetilde{U}^+_{1,j}$ or does not intersect any $\widetilde{U}^+_{1,j}$. 
Select the rectangles $\widehat{U}^+_{2,j}$, that do not intersect any $\widetilde{U}^+_{1,j}$, and denote the obtained collection by $\{ \widetilde{U}^+_{2,j} \}_j$.
At the $i$th step, choose those $\widehat{U}^+_{i,j}$ that do not intersect any previously selected $\widetilde{U}^+_{i',j}$, $i' < i$.
Hence, we obtain a collection $\{ \widetilde{U}^+_{i,j} \}_{i,j}$ of pairwise disjoint rectangles.
Observe that for every $U^+_{i,j}$ there exists $\widetilde{U}^+_{i',j}$ with $i' < i$ such that
\begin{equation}
\label{plussubset}
\text{pr}_x(U^+_{i,j}) \subset \text{pr}_x(\widetilde{U}^+_{i',j}) \quad \text{and} \quad \text{pr}_t(U^+_{i,j}) \subset 3 \text{pr}_t(\widetilde{U}^+_{i',j}) .
\end{equation}
Note that $S^-_{i,j}$ is spatially contained in $U^+_{i,j}$, that is, $\text{pr}_x S^-_{i,j}\subset \text{pr}_x U^+_{i,j}$.
In the time direction, we have
\begin{equation}
\label{minusplussubset}
\text{pr}_t(S^-_{i,j}) 
\subset ( 3+2\tau ) \text{pr}_t(U^+_{i,j}) ,
\end{equation}
since
\[
( 4 + 2\tau ) \frac{l_t(U^+_{i,j})}{2} 
=
(2+\tau) L_i^p .
\]
Therefore, by~\eqref{plussubset} and~\eqref{minusplussubset}, it holds that
\begin{equation}
\label{subsetcubes}
\sum_{i,j} \lvert S^-_{i,j} \rvert 
= 
\Big\lvert \bigcup_{i,j} S^-_{i,j} \Big\rvert 
\leq 
c_1 \sum_{i,j} \lvert \widetilde{U}^+_{i,j} \rvert 
\quad\text{with}\quad
c_1 = 3 ( 3+2\tau ).
\end{equation}

Let $\sigma = 2^{n+p}\beta$.
It holds that
\begin{align*}
w(U^+_{i,j} \cap \{ w \leq (1-\sigma) w_{U^+_{i,j}} \})
\leq
(1-\sigma) w_{U^+_{i,j}} \lvert U^+_{i,j} \rvert
=
(1-\sigma) w(U^+_{i,j})
\end{align*}
from which we obtain
\[
w(U^+_{i,j} \cap \{ w > (1-\sigma) w_{U^+_{i,j}} \}) \geq \sigma w(U^+_{i,j}) .
\]
From the selection criterion \eqref{eq:selectcrit1}, we get
\[
w(U^+_{i-1,j}) \leq \lambda \lvert U^+_{i-1,j} \rvert = 2^{n+p} \lambda \lvert U^+_{i,j} \rvert < 2^{n+p} w(U^+_{i,j}) .
\]
By the last two estimates,
we have
\[
w(U^+_{i,j} \cap \{ w > (1-\sigma) w_{U^+_{i,j}} \}) > \frac{\sigma}{2^{n+p}} w(U^+_{i-1,j}) = \beta w(U^+_{i-1,j}) .
\]
Recall that $U^+_{i,j} \subset U^-_{i-1,j}$.
Thus, we may apply Lemma~\ref{lemma:timemove}~$(i)$ to obtain
\[
\lvert U^+_{i,j} \cap \{ w > (1-\sigma) w_{U^+_{i,j}} \} \rvert \geq \alpha \lvert U^+_{i-1,j} \rvert
\]
and since $w_{U^+_{i,j}}>\lambda$ we have
\begin{equation}
\label{superlevelestimate}
\lvert U^+_{i,j} \cap \{ w > (1-\sigma) \lambda \} \rvert
\geq
\lvert U^+_{i,j} \cap \{ w > (1-\sigma) w_{U^+_{i,j}} \} \rvert \geq \alpha \lvert U^+_{i-1,j} \rvert .
\end{equation}

If $(x,t) \in S^-_0 \setminus \bigcup_{i,j} S^-_{i,j}$, then there exists a sequence of subrectangles $S^-_l$ containing $(x,t)$ such that 
\[
\frac{w(U^+_l)}{\lvert U^+_l \rvert} = \dashint_{U^+_l} w \leq \lambda
\]
and $\lvert S^-_l \rvert \to 0$ as $l \to \infty$.
The Lebesgue differentiation theorem~\cite[Lemma~2.3]{KinnunenMyyryYang2022}
implies that $w(x,t) \leq \lambda$
for almost every $(x,t) \in S^-_0 \setminus \bigcup_{i,j} S^-_{i,j}$.
It follows that
\[
S^-_0 \cap \{ w > \lambda \} \subset \bigcup_{i,j} S^-_{i,j}
\]
up to a set of measure zero.
By using this with Lemma~\ref{lemma:timemove}~$(ii)$ for $\theta=1+\tau$, the selection criterion~\eqref{eq:selectcrit1},~\eqref{subsetcubes} and~\eqref{superlevelestimate}, 
we obtain
\begin{align*}
w( S^-_0 \cap \{ w > \lambda \} ) 
&\leq 
\sum_{i,j} w( S^-_{i,j} )
\leq 
\sum_{i,j} w( R^-_{i-1,j} )
\leq 
C \sum_{i,j} w( U^+_{i-1,j} )\\
&\leq 
C \lambda \sum_{i,j} \lvert U^+_{i-1,j} \rvert 
\leq
2^{n+p+1} C \lambda \sum_{i,j} \lvert S^-_{i,j} \rvert \\
&\leq
2^{n+p+1} c_1 C \lambda \sum_{i,j} \lvert \widetilde{U}^+_{i,j} \rvert 
\leq
2 c_1 C \lambda \sum_{i,j} \lvert \widetilde{U}^+_{i-1,j} \rvert 
\\
&\leq
2 c_1 C \alpha^{-1} \lambda \sum_{i,j} \lvert \widetilde{U}^+_{i,j} \cap \{ w > (1-\sigma) \lambda \} \rvert
\\
&\leq
2 c_1 C \alpha^{-1} \lambda \lvert R_0 \cap \{ w > (1-\sigma) \lambda \} \rvert .
\end{align*}
This completes the proof.
\end{proof}

The following theorem states that the qualitative measure condition implies the parabolic reverse H\"older inequality.

\begin{theorem}
Let $w$ be a weight.
Assume that there exist 
$0<\alpha<1$ and $0<\beta<1/2^{n+p}$
such that for every parabolic rectangle $R$ and every measurable set $E \subset R^-$ for which $\lvert E \rvert < \alpha \lvert R^- \rvert$ we have $w(E) < \beta w(R^+)$.
Then $w\in RH^+_q$ for some $1<q<\infty$.

\end{theorem}

\begin{proof}
Let $R \subset \mathbb R^{n+1}$ be a parabolic rectangle.
Let $\varepsilon>0$ to be chosen later.
We use the same notation as in the statement of Lemma~\ref{weightmeasure-estimate}.
Hence, 
for $\lambda \geq w_{U^+}$ we have
\[
w( R^{-} \cap \{ w > \lambda \}  )  \leq c \lambda \lvert R \cap \{ w > \sigma \lambda \} \rvert ,
\]
where $\sigma=1-2^{n+p}\beta$ and
$U^+ = R^+ + (0, \tau L^p)$ with 
$\tau = 1/(2^p-1)$. 
We show that this implies the corresponding inequality for the truncated weight $w_k = \min\{w,k\}$, $k\in\mathbb{Z}$, that is,
\begin{equation}
\label{eq:qualiRH_trunc}
w( R^{-} \cap \{ w_k > \lambda \}  )  \leq c \lambda \lvert R \cap \{ w_k > \sigma \lambda \} \rvert .
\end{equation}
If $\lambda\geq k$, then $\{ w_k > \lambda \} = \emptyset $ and thus the estimate holds.
On the other hand,
if $\lambda < k$,
then $\{ w_k > \lambda \} = \{ w > \lambda \}$ and $\{ w_k > \sigma \lambda \} = \{ w > \sigma \lambda \}$.
Hence, \eqref{eq:qualiRH_trunc} holds true.

Applying \eqref{eq:qualiRH_trunc} with Cavalieri's principle and Lemma~\ref{lemma:timemove}~$(ii)$ for $\theta=1+\tau$ (with the constant $C$), we obtain
\begin{align*}
\int_{R^-} w_k^{1+\varepsilon} &\leq \varepsilon \int_0^\infty \lambda^{\varepsilon-1} w( R^{-} \cap \{ w_k > \lambda \}  ) \dla \\
&= \varepsilon \int_0^{w_{U^+}} \lambda^{\varepsilon-1} w( R^{-} \cap \{ w_k > \lambda \}  ) \dla\\
&\qquad + \varepsilon \int_{w_{U^+}}^\infty \lambda^{\varepsilon-1} w( R^{-} \cap \{ w_k > \lambda \}  ) \dla \\
&\leq w( R^{-} ) \varepsilon \int_0^{w_{U^+}} \lambda^{\varepsilon-1} \dla + c \varepsilon \int_{w_{U^+}}^\infty \lambda^{\varepsilon} \lvert R \cap \{ w_k > \sigma \lambda \} \rvert \dla \\
&\leq w( R^{-} ) w_{U^+}^\varepsilon + \frac{c \varepsilon}{\sigma^{1+\varepsilon}}  \int_0^\infty \lambda^{\varepsilon} \lvert R \cap \{ w_k > \lambda \} \rvert \dla \\
&\leq C \lvert U^{+} \rvert w_{U^+}^{1+\varepsilon} + \frac{c}{\sigma^{1+\varepsilon}} \frac{\varepsilon}{1+\varepsilon} \int_{R} w_k^{1+\varepsilon} .
\end{align*}
By choosing $\varepsilon>0$ to be small enough, we can absorb the integral over $R^-$ of the second term to the left-hand side to get
\begin{align*}
\biggl( 1- \frac{c}{\sigma^{1+\varepsilon}} \frac{\varepsilon}{1+\varepsilon} \biggr) \int_{R^-} w_k^{1+\varepsilon} \leq C \lvert U^{+} \rvert w_{U^+}^{1+\varepsilon} + \frac{c}{\sigma^{1+\varepsilon}} \frac{\varepsilon}{1+\varepsilon} \int_{R^+} w_k^{1+\varepsilon} .
\end{align*}
Hence, we have
\begin{equation}
\label{cavalieri_iteration}
\int_{R^-} w_k^{1+\varepsilon} \leq c_0 \lvert U^{+} \rvert w_{U^+}^{1+\varepsilon} + c_1 \varepsilon \int_{R^+} w_k^{1+\varepsilon} ,
\end{equation}
where
\[
c_0 = \frac{C(1+\varepsilon) }{1-(c \sigma^{-1-\varepsilon} -1) \varepsilon} 
\quad \text{and} \quad 
c_1 = \frac{c \sigma^{-1-\varepsilon} }{1-(c \sigma^{-1-\varepsilon} -1) \varepsilon} .
\]

Fix $R_0 = Q(x_0,L) \times (t_0 - L^p, t_0 + L^p) \subset \mathbb R^{n+1}$.
We cover $R^-_0$ by $M = 2^{n+1}$ rectangles $R^-_{1,j}$ with spatial side length $l_x = L/2^{1/p}$ and time length $l_t = L^p / 2$. This can be done by dividing each spatial edge of $R^-_0$ into two equally long intervals that may overlap each other, and the time interval of $R^-_0$ into two equally long pairwise disjoint intervals.
Observe that the overlap of $R^-_{1,j}$ is bounded by $M/2 = 2^n$.
Then consider $R^+_{1,j}$ and cover it in the same way as before by $M$ rectangles $R^-_{2,j}$ with spatial side length $l_x = L/2^{2/p}$ and time length $l_t = L^p / 2^2$.
At the $i$th step, cover $R^+_{i-1,j}$ by $M$ rectangles $R^-_{i,j}$ with spatial side length $l_x = L/2^{i/p}$ and time length $l_t = L^p / 2^i$ such that their overlap is bounded by $M/2$.
We note that every $R_{i,j}$ and corresponding $U^+_{i,j}$ is contained in $R_0$.
By iterating~\eqref{cavalieri_iteration} we obtain
\begin{align*}
\int_{R^-_0} w_k^{1+\varepsilon} &\leq \sum_{j=1}^{M} \int_{R^-_{1,j}} w_k^{1+\varepsilon}  
\leq 
\sum_{j=1}^{M} c_0 \lvert U^{+}_{1,j} \rvert w_{U^+_{1,j}}^{1+\varepsilon} + \sum_{j=1}^{M} c_1 \varepsilon \int_{R^+_{1,j}} w_k^{1+\varepsilon} \\
&\leq 
c_0 \sum_{j=1}^M \lvert U^{+}_{1,j} \rvert w_{U^+_{1,j}}^{1+\varepsilon} + c_1 \varepsilon \sum_{j=1}^{M^2} \int_{R^-_{2,j}} w_k^{1+\varepsilon} \\
&\leq 
c_0 \sum_{j=1}^M \lvert U^{+}_{1,j} \rvert w_{U^+_{1,j}}^{1+\varepsilon} + c_1 \varepsilon \sum_{j=1}^{M^2} \biggl( c_0 \lvert U^+_{2,j}\rvert w_{U^+_{2,j}}^{1+\varepsilon} + c_1 \varepsilon \int_{R^+_{2,j}} w_k^{1+\varepsilon} \biggr) \\
&= 
c_0 \sum_{j=1}^M \lvert U^{+}_{1,j} \rvert w_{U^+_{1,j}}^{1+\varepsilon} + c_0 c_1 \varepsilon \sum_{j=1}^{M^2} \lvert U^+_{2,j}\rvert w_{U^+_{2,j}}^{1+\varepsilon} + (c_1 \varepsilon)^2 \sum_{j=1}^{M^2} \int_{R^+_{2,j}} w_k^{1+\varepsilon} \\
&\leq 
c_0 \sum_{i=1}^N \biggl( (c_1 \varepsilon)^{i-1} \sum_{j=1}^{M^i} \lvert U^+_{i,j} \rvert w^{1+\varepsilon}_{U^+_{i,j}} \biggr) + (c_1 \varepsilon)^N \sum_{j=1}^{M^N} \int_{R^+_{N,j}} w_k^{1+\varepsilon} \\
&\leq c_0 \sum_{i=1}^N \biggl( (c_1 \varepsilon)^{i-1} \sum_{j=1}^{M^i} \lvert U^+_{i,j} \rvert w^{1+\varepsilon}_{U^+_{i,j}} \biggr) + \biggl( c_1 \varepsilon \frac{M}{2} \biggr)^N \int_{R_0} w_k^{1+\varepsilon} \\
&= I + II .
\end{align*}
We observe that $II$ tends to zero if $\varepsilon < \frac{2}{c_1 M} = \frac{1}{c_1 2^n}$ as $N \to \infty$.
Since
\[
\lvert U^+_{i,j} \rvert^{-\varepsilon} = L^{-(n+p)\varepsilon} 2^{(\frac{n}{p}+1)i \varepsilon } =  2^{1+\varepsilon} L^{n+p} 2^{(\frac{n}{p}+1)i \varepsilon } \lvert R_{0} \rvert^{-(1+\varepsilon)},
\]
for the inner sum of the first term $I$ we have 
\begin{align*}
\sum_{j=1}^{M^i} \lvert U^+_{i,j} \rvert w^{1+\varepsilon}_{U^+_{i,j}} 
&= \sum_{j=1}^{M^i} \lvert U^+_{i,j} \rvert^{-\varepsilon} \biggl( \int_{U^+_{i,j}} w \biggr)^{1+\varepsilon}\\
&\leq 2^{1+\varepsilon} L^{n+p} 2^{(\frac{n}{p}+1)i \varepsilon } \biggl(\frac M2\biggr)^i w^{1+\varepsilon}_{R_0} .
\end{align*}
Thus, it follows that
\begin{align*}
I \leq c_0 2^{1+\varepsilon} L^{n+p} w^{1+\varepsilon}_{R_0}  \sum_{i=1}^N  (c_1 \varepsilon)^{i-1} 2^{(\frac{n}{p}+1)i \varepsilon } \biggl(\frac M2\biggr)^i .
\end{align*}
We estimate the sum by
\begin{align*}
\sum_{i=1}^N  (c_1 \varepsilon)^{i-1} 2^{(\frac{n}{p}+1)i \varepsilon } \biggl(\frac M2\biggr)^i
&= 2^{(\frac{n}{p}+1) \varepsilon } \frac{M}{2} \sum_{i=0}^{N-1} \biggl( c_1 \varepsilon 2^{(\frac{n}{p}+1) \varepsilon } \frac{M}{2} \biggr)^i \\
&\leq 2^{(\frac{n}{p}+1) \varepsilon } \frac{M}{2} \frac{1}{1-c_1 \varepsilon 2^{(\frac{n}{p}+1) \varepsilon } \frac{M}{2} } \\
&= \frac{2^{(\frac{n}{p}+1) \varepsilon +n }}{1-c_1 \varepsilon 2^{(\frac{n}{p}+1) \varepsilon + n } } ,
\end{align*}
whenever $\varepsilon$ is small enough, for example
\[
\varepsilon < \frac{1}{c_1 2^{\frac{n}{p}+1} M } = \frac{1}{c_1 2^{\frac{n}{p}+1+n} } .
\]
Then it holds that
\begin{align*}
\int_{R^-_0} w_k^{1+\varepsilon} &\leq c_0 2^{1+\varepsilon} L^{n+p} w^{1+\varepsilon}_{R_0} \frac{ 2^{(\frac{n}{p}+1) \varepsilon +n } }{1- c_1 \varepsilon 2^{(\frac{n}{p}+1) \varepsilon +n } }
\end{align*}
for small enough $\varepsilon$.
Since $w_{R^-_0} \leq C w_{R^+_0}$ for some $C=C(\alpha,\beta)$
by \eqref{eq:lowerupperpart_bound} in the proof of Lemma~\ref{lemma:timemove}~$(ii)$,
we conclude that
\begin{align*}
\biggl( \dashint_{R^-_0} w_k^{1+\varepsilon} \biggr)^\frac{1}{1+\varepsilon} &\leq c_2 \dashint_{R_0} w = \frac{c_2}{2} \dashint_{R^-_0} w + \frac{c_2 }{2} \dashint_{R^+_0} w \leq \frac{c_2 }{2} (C+1) \dashint_{R^+_0} w ,
\end{align*}
where
\[
c_2 = 2 \biggl( c_0 \frac{ 2^{(\frac{n}{p}+1) \varepsilon +n } }{1- c_1 \varepsilon 2^{(\frac{n}{p}+1) \varepsilon +n}} \biggr)^\frac{1}{1+\varepsilon} .
\]
The claim follows from the monotone convergence theorem as $k \to \infty$.
\end{proof}

\subsection{Superlevel measure condition}
We show $(ii)\Rightarrow (v) \Rightarrow (iv)$ in Theorem~\ref{thm:RHIchar}.
We start with $(ii)$ implies $(v)$.

\begin{theorem}
Let $w$ be a weight.
Assume that
there exist constants $C,\delta >0$ such that
\[
\frac{w(E)}{w(R^+)} 
\leq
C \biggl( \frac{\lvert E \rvert}{\lvert R^- \rvert} \biggr)^\delta
\]
for every parabolic rectangle $R\subset\mathbb{R}^{n+1}$ and 
measurable set $E \subset R^-$.
Then there exist $0<\alpha<1$ and $0<\beta<1/2^{n+p}$ such that for every parabolic rectangle $R$ we have
\[
w( R^- \cap \{ \alpha w > w_{R^+} \} ) < \beta w( R^+ ) .
\]
\end{theorem}

\begin{proof}

Denote $E = R^- \cap \{ \alpha w > w_{R^+} \} $.
We have $\lvert E \rvert < \alpha w(E)/w_{R^+}$.
Thus, 
the assumption
implies that
\[
\frac{w(E)}{w(R^+)} 
\leq
C \biggl(  \frac{\lvert E \rvert}{\lvert R^- \rvert} \biggr)^\delta 
<
C \biggl( \alpha \frac{w(E)}{w(R^+)} \biggr)^\delta ,
\]
from which 
we get
\[
w(E) < C^\frac{1}{1-\delta} \alpha^\frac{\delta}{1-\delta} w(R^+) .
\]
We finish the proof by
choosing $\alpha$ small enough such that
\[
\beta = C^\frac{1}{1-\delta} \alpha^\frac{\delta}{1-\delta} < \frac{1}{2^{n+p}} .
\]
\end{proof}

Next we show 
that
$(v)$ implies $(iv)$
in Theorem~\ref{thm:RHIchar}.

\begin{theorem}
Let $w$ be a weight.
Assume that
there exist $0<\alpha<1$ and $0<\beta<1/2^{n+p}$ such that for every parabolic rectangle $R$ we have
\[
w( R^- \cap \{ \alpha w > w_{R^+} \} ) < \beta w( R^+ ) .
\]
Then there exist $0<\alpha'<1$ and $0<\beta'<1/2^{n+p}$ such that for every parabolic rectangle $R$ and every measurable set $E \subset R^-$ for which $\lvert E \rvert < \alpha' \lvert R^- \rvert$ we have $w(E) < \beta' w(R^+)$.

\end{theorem}

\begin{proof}
Let $E \subset R^-$ be a measurable set such that $ \lvert E \rvert < \alpha' \lvert R^- \rvert $ where $\alpha' < (1/2^{n+p}-\beta)\alpha$.
It follows that
\begin{align*}
w( E ) &\leq w( E \cap \{\alpha w > w_{R^+} \} ) + w( E \cap \{\alpha w \leq w_{R^+} \} ) \\
&\leq \beta w( R^+ ) + \frac{w_{R^+}}{\alpha} \lvert E \rvert 
= \biggl( \beta + \frac{1}{\alpha} \frac{\lvert E \rvert}{\lvert R^+ \rvert} \biggr) w(R^+) \\
&< \biggl( \beta + \frac{\alpha'}{\alpha} \biggr) w(R^+) 
= \beta' w(R^+),
\end{align*}
where $\beta' = \beta + \frac{\alpha'}{\alpha} < 1/2^{n+p}$.
\end{proof}

\subsection{Fujii--Wilson condition}
We show $(i)\Rightarrow (vi) \Rightarrow (vii) \Rightarrow (iii)$ in Theorem~\ref{thm:RHIchar}.
We begin with the boundedness of the parabolic maximal function on $L^q$.

\begin{lemma}
\label{lem:MfbddLq}
Let $1<q\leq \infty$.
Assume that  $f\in L^1_{\mathrm{loc}}(\mathbb{R}^{n+1})$.
Then there exists a constant $c$ such that
\begin{align*}
\int_{\mathbb{R}^{n+1}} (M^+f)^q \leq
c \int_{\mathbb{R}^{n+1}} \lvert f \rvert^q .
\end{align*}

\end{lemma}

\begin{proof}
Let $E = \{ M^+f > \lambda \}$.
For every $z \in E$ there exists a parabolic rectangle $R_{z}$ such that 
$z \in R^-_{z}$ and
\[
\dashint_{R^+_{z}} \lvert f \rvert > \lambda .
\]
By a similar argument to the Vitali covering theorem, 
we obtain
a countable collection $\{R_i\}_i$ of pairwise disjoint parabolic rectangles such that
\[
E \subset \bigcup_{z\in E} R_z \subset \bigcup_{i=1}^\infty 5R_i .
\]
Thus, we have
\begin{align*}
\lvert E \rvert &\leq \sum_i \lvert 5 R_i \rvert = 5^{n+p} \sum_i \lvert R_i \rvert 
= 5^{n+p} 2 \sum_i \lvert R_i^+ \rvert \\
&\leq
\frac{5^{n+p} 2}{\lambda} \sum_i \int_{R^+_{i}} \lvert f \rvert
\leq
\frac{5^{n+p} 2}{\lambda} \int_{\mathbb{R}^{n+1}} \lvert f \rvert .
\end{align*}
In other words, $M^+f$ is bounded from $L^{1}$ to $L^{1,\infty}$.
Moreover, we observe that $M^{+}f$ is bounded 
on $L^{\infty}$ 
since
\[
\norm{M^{+} f}_{L^\infty(\mathbb{R}^{n+1})} \leq \norm{f}_{L^\infty(\mathbb{R}^{n+1})} .
\]
The Marcinkiewicz interpolation theorem implies that $M^{+}f$ is bounded 
on $L^{q}$,
particularly
\begin{align*}
\int_{\mathbb{R}^{n+1}} (M^{+}f)^q &\leq
\frac{ q 2^{q+1}  5^{n+p} }{q-1}  
\int_{\mathbb{R}^{n+1}} \lvert f \rvert^q .
\end{align*}
\end{proof}

The next theorem states that the parabolic reverse H\"older inequality implies the parabolic Fujii--Wilson condition.

\begin{theorem}
Let $1<q<\infty$.
Assume that $w\in RH^+_q$.
Then there exists a constant $C$ such that
\[
\int_{R^-} M^+ (w \chi_{R^-}) 
\leq 
C \int_{R^+} w
\]
for every parabolic rectangle $R\subset\mathbb{R}^{n+1}$.
\end{theorem}

\begin{proof}
By H\"older's inequality, Lemma~\ref{lem:MfbddLq} (with the constant $c$) and the assumption, we obtain
\begin{align*}
\dashint_{R^-} M^+(w \chi_{R^-}) 
&\leq 
\biggl( \dashint_{R^-} M^+(w \chi_{R^-})^q \biggr)^\frac{1}{q} \\
&\leq
c \biggl( \dashint_{R^-} w^q \biggr)^\frac{1}{q} 
\leq 
c C \dashint_{R^+} w .
\end{align*}
This completes the proof.
\end{proof}

The following lemma is a reverse weak type estimate for the parabolic maximal function.

\begin{lemma}
\label{lem:reverseweaktype}
Let $w$ be a weight.
Assume that there exists a constant $C$ such that $w(R^-)\leq C w(R^+)$ for every parabolic rectangle $R\subset\mathbb{R}^{n+1}$.
Then there exists a constant $c$ such that
for every parabolic rectangle $R \subset \mathbb{R}^{n+1}$ and $\lambda \geq w_{R^+}$ we have
\[
w(R^- \cap \{w>\lambda\}) \leq c \lambda \lvert R^- \cap \{ M^+ w > \lambda \}\rvert .
\]
\end{lemma}

\begin{proof}

Let $R_0=R(x_0,t_0,L) = Q(x_0,L) \times (t_0-L^p, t_0+L^p)$
and $\lambda \geq w_{R^+_0}$.
Denote $S^-_0 = R^-_0$. The time length of $S^-_0$ is $l_t(S^-_0) = L^p$.
We construct a parabolic Calder\'on--Zygmund decomposition at level $\lambda$.
We partition $S^-_0$ by dividing each spatial edge into $2$ equally long intervals. If
\[
\frac{l_t(S_{0}^-)}{\lfloor 2^{p} \rfloor} < \frac{L^p}{2^{p}},
\]
we divide the time interval of $S^-_0$ into $\lfloor 2^{p} \rfloor$ equally long intervals. Otherwise, we divide the time interval of $S^-_0$ into $\lceil 2^{p} \rceil$ equally long intervals.
We obtain subrectangles $S^-_1$ of $S^-_0$ with spatial side length $L_1 = l_x(S^-_1) = l_x(S^-_0)/2 = L / 2$
and time length either
\[
l_t(S^-_1) = \frac{l_t(S^-_0)}{\lfloor 2^{p} \rfloor} = \frac{L^p}{\lfloor 2^{p} \rfloor} \quad \text{or} \quad l_t(S^-_1) = \frac{L^p}{\lceil 2^{p} \rceil} .
\]
For every $S^-_1$, there exists a unique rectangle $R_1$ with spatial side length $L_1 = L / 2$ 
and time length $L_1^p = 2 L^p / 2^{p}$
such that $R_1$ has the same bottom as $S^-_1$.
We select those rectangles $S^-_1$ for which
\[
\frac{w(R^+_1)}{\lvert R^+_1 \rvert} =
\dashint_{R^+_1} w > \lambda
\]
and denote the obtained collection by $\{ S^-_{1,j} \}_j$.
If
\[
\frac{w(R^+_1)}{\lvert R^+_1 \rvert} =
\dashint_{R^+_1} w \leq \lambda ,
\]
we subdivide $S^-_1$ in the same manner as above
and select all those subrectangles $S^-_2$ for which
\[
\frac{w(R^+_2)}{\lvert R^+_2 \rvert} =
\dashint_{R^+_2} w  > \lambda
\]
to obtain family $\{ S^-_{2,j} \}_j$.
We continue this selection process recursively.
At the $i$th step, we partition unselected rectangles $S^-_{i-1}$ by dividing each spatial side into $2$ equally long intervals. If
\begin{equation}
\label{weaktype:JNproof_eq1}
\frac{l_t(S_{i-1}^-)}{\lfloor 2^{p} \rfloor} < \frac{L^p}{2^{pi}},
\end{equation}
we divide the time interval of $S^-_{i-1}$ into $\lfloor 2^{p} \rfloor$ equally long intervals. 
Otherwise, if
\begin{equation}
\label{weaktype:JNproof_eq2}
\frac{l_t(S_{i-1}^-)}{\lfloor 2^{p} \rfloor} \geq \frac{L^p}{2^{pi}},
\end{equation}
we divide the time interval of $S^-_{i-1}$ into $\lceil 2^{p} \rceil$ equally long intervals.
We obtain subrectangles $S^-_i$. 
For every $S^-_i$, there exists a unique rectangle $R_i$ with spatial side length $L_i = L / 2^{i}$
and time length $L_i^p = 2 L^p / 2^{pi}$
such that $R_i$ has the same bottom as $S^-_i$.
Select those $S^-_i$ for which 
\[
\frac{w(R^+_i)}{\lvert R^+_i \rvert} =
\dashint_{R^+_i} w  > \lambda
\]
and denote the obtained collection by $\{ S^-_{i,j} \}_j$.
If 
\[
\frac{w(R^+_i)}{\lvert R^+_i \rvert} =
\dashint_{R^+_i} w \leq \lambda ,
\]
we continue the selection process in $S^-_i$.
In this manner we obtain a collection $\{S^-_{i,j} \}_{i,j}$ of pairwise disjoint rectangles.

If $(x,t) \in S^-_0 \setminus \bigcup_{i,j} S^-_{i,j}$, then there exists a sequence of subrectangles $S^-_l$ containing $(x,t)$ such that 
\[
\frac{w(R^+_l)}{\lvert R^+_l \rvert} =
\dashint_{R^+_l} w \leq \lambda
\]
and $\lvert S^-_l \rvert \to 0$ as $l \to \infty$.
The Lebesgue differentiation theorem~\cite[Lemma~2.3]{KinnunenMyyryYang2022}
implies that $w(x,t) \leq \lambda$ for almost every $(x,t) \in S^-_0 \setminus \bigcup_{i,j} S^-_{i,j}$.
It follows that
\[
S^-_0 \cap \{ w > \lambda \} \subset \bigcup_{i,j} S^-_{i,j}
\]
up to a set of measure zero.

By the assumption, we have $w(R^-_{i-1,j}) \leq C w(R^+_{i-1,j})$ for every $R_{i-1,j}$.
Since
\[
\lambda < \dashint_{R^+_{i,j}} w  \leq M^+ w(x,t)
\]
for every $(x,t)\in S^-_{i,j} \subset R^-_{i,j}$, by \eqref{eq:tsidel} we conclude that
\begin{align*}
w(S^-_0 \cap \{ w > \lambda \})
&\leq \sum_{i,j} w(S^-_{i,j})
\leq \sum_{i,j} w(R^-_{i-1,j})
\leq C \sum_{i,j} w(R^+_{i-1,j})\\
&\leq C \lambda \sum_{i,j} \lvert {R^+_{i-1,j}} \rvert
\leq 2^{n+p+1} C \lambda \sum_{i,j} \lvert {S^-_{i,j}} \rvert \\
&= 2^{n+p+1} C \lambda \sum_{i,j} \lvert {S^-_{i,j}} \cap \{M^+ w(x,t) > \lambda\} \rvert \\
&\leq
2^{n+p+1} C \lambda \lvert S^-_0 \cap \{M^+ w(x,t) > \lambda\} \rvert .
\end{align*}
This completes the proof.
\end{proof}

We observe that the parabolic Fujii--Wilson condition implies the following parabolic logarithmic condition.

\begin{theorem}
Let $w$ be a weight.
Assume that there exists a constant $C_1$ such that
\[
\int_{R^-} M^+ (w \chi_{R^-}) 
\leq 
C_1 \int_{R^+} w
\]
for every parabolic rectangle $R\subset\mathbb{R}^{n+1}$.
Then there exists a constant $C_2$ such that
\[
\int_{R^-} w \log^+ \biggl(\frac{w}{w_{R+}}\biggr) \leq C_2 w(R^+)
\]
for every parabolic rectangle $R\subset\mathbb{R}^{n+1}$.
\end{theorem}

\begin{proof}
Since the assumption implies $w(R^-)\leq C_1 w(R^+)$ for every parabolic rectangle $R\subset\mathbb{R}^{n+1}$, we observe that Lemma~\ref{lem:reverseweaktype} 
is applicable.
Thus, 
it follows that
\begin{align*}
\int_{R^-} w \log^+ \biggl(\frac{w}{w_{R+}}\biggr)
&= 
\int_{R^-\cap \{w>w_{R+}\}} \biggl(w \int_{w_{R^+}}^{w} \frac{1}{\lambda} \dla\biggr)\\
&=
\int_{w_{R^+}}^\infty\biggl( \frac{1}{\lambda} \int_{R^- \cap \{w>\lambda\}} w\biggr)\dla \\
&=
\int_{w_{R^+}}^\infty \frac{1}{\lambda} w(R^- \cap \{w>\lambda\}) \dla \\
&\leq
c \int_{w_{R^+}}^\infty \lvert R^- \cap \{ M^+ (w \chi_{R^-}) > \lambda \}\rvert \dla
\\
&\leq
c \int_{R^-} M^+ (w \chi_{R^-})
\leq c C_1 \int_{R^+} w .
\end{align*}
\end{proof}

The next theorem shows that the parabolic logarithmic condition implies the qualitatitive measure condition.
This completes the proof of Theorem~\ref{thm:RHIchar}.

\begin{theorem}
Let $w$ be a weight.
Assume that there exists a constant $C$ such that
\[
\int_{R^-} w \log^+ \biggl(\frac{w}{w_{R+}}\biggr) \leq C w(R^+)
\]
for every parabolic rectangle $R\subset\mathbb{R}^{n+1}$.
Then 
for every $\beta>0$ there exists $0<\alpha<1$ such that for every parabolic rectangle $R$ and every measurable set $E \subset R^-$ for which $\lvert E \rvert < \alpha \lvert R^- \rvert$ we have $w(E) < \beta w(R^+)$.
\end{theorem}

\begin{proof}
Let $\beta>0$. Choose $\sigma>1$ such that $C/\log \sigma \leq \beta/2$ and $0<\alpha<1$ such that $\sigma \alpha \leq \beta/2$.
Let $E\subset R^-$ be a measurable set with $\lvert E \rvert < \alpha \lvert R^- \rvert$.
Then we have
\begin{align*}
w(E \cap \{ w \leq \sigma w_{R^+} \})
\leq
\sigma w_{R^+} \lvert E \rvert 
<
\sigma \alpha w(R^+)
\leq
\frac{\beta}{2} w(R^+)
\end{align*}
and
\begin{align*}
w(E \cap \{ w > \sigma w_{R^+} \}) 
&=
\frac{1}{\log\sigma} \int_{E \cap \{w > \sigma w_{R^+}\}} w \log \sigma\\
&\leq
\frac{1}{\log\sigma} \int_{E \cap \{w > \sigma w_{R^+}\}} w \log \biggl(\frac{w}{w_{R^+}}\biggr)
\\
&\leq
\frac{1}{\log\sigma} \int_{R^- \cap \{w > w_{R^+}\}} w \log \biggl(\frac{w}{w_{R^+}}\biggr)\\
&=
\frac{1}{\log\sigma} \int_{R^-} w \log^+ \biggl(\frac{w}{w_{R^+}}\biggr)
\\
&\leq
\frac{C}{\log\sigma} w(R^+)
\leq
\frac{\beta}{2} w(R^+) .
\end{align*}
This shows that $w(E) < \beta w(R^+)$.
\end{proof}

\section{Parabolic Gehring lemma}

In this section, we show the parabolic Gehring lemma which states that the parabolic reverse H\"older inequality is self-improving. 
In particular, it implies that if $w\in RH^+_q$, then $w \in RH^+_{q+\varepsilon}$ for some $\varepsilon>0$.
The results in this section also hold in the case $p=1$.
The next lemma is the main ingredient in the proof of the parabolic Gehring lemma.

\begin{lemma}
\label{selfimprovelemma}
Let $1<q<\infty$ and $w$ be a weight.
Assume that there exists a constant $C_1>1$ such that for every parabolic rectangle $R\subset\mathbb{R}^{n+1}$ and $\lambda\geq w_{R^+}$ we have
\[
\int_{R^- \cap \{w>\lambda\}} w^q \leq C_1 \lambda^{q-1} \int_{R \cap \{w>\lambda\}} w .
\]
Then there exist $\varepsilon=\varepsilon(n,p,q,C_1)>0$ and $C=C(n,p,q,C_1)$ such that for every $R\subset\mathbb{R}^{n+1}$ we have
\[
\int_{R^-} w^{q+\varepsilon} \leq C \biggl( \dashint_{R} w \biggr)^\varepsilon \int_{R} w^q .
\]
\end{lemma}

\begin{proof}
Let $R \subset \mathbb{R}^{n+1}$ be a parabolic rectangle and $\lambda_0 = w_{R^+}$.
Let $\varepsilon>0$ to be chosen later.
We show that the assumption implies the corresponding estimate for the truncated weight $w_k = \min\{w,k\}$, $k\in\mathbb{Z}$, that is,
\begin{equation}
\label{eq:gehring_trunc}
\int_{R^- \cap \{w_k>\lambda\}} w^q \leq C_1 \lambda^{q-1} \int_{R \cap \{w_k>\lambda\}} w .
\end{equation}
If $\lambda\geq k$, then $\{ w_k > \lambda \} = \emptyset $ and thus the estimate holds.
On the other hand,
if $\lambda < k$,
then $\{ w_k > \lambda \} = \{ w > \lambda \}$.
Hence, \eqref{eq:gehring_trunc} holds true.

We apply Cavalieri's principle with the exponent $\varepsilon$ and the measure $d\mu = w^q \dx\dt$ to obtain
\begin{align*}
&\int_{R^- \cap \{w_k>\lambda_0\}} w_k^{q+\varepsilon -1} w 
\leq \int_{R^- \cap \{w_k>\lambda_0\}} w_k^{\varepsilon} \dmu \\ 
&\qquad\leq \varepsilon \int_{\lambda_0}^\infty \biggl(\lambda^{\varepsilon -1} \int_{R^- \cap \{w_k>\lambda\}} w^q\biggr) \dla + \lambda_0^{\varepsilon} \int_{R^- \cap \{w_k>\lambda_0\}} w^q .
\end{align*}
The estimate \eqref{eq:gehring_trunc} implies
\[
\int_{\lambda_0}^\infty\biggl(\lambda^{\varepsilon -1} \int_{R^- \cap \{w_k>\lambda\}} w^q\biggr) \dla \leq C_1 \int_{\lambda_0}^\infty \biggl(\lambda^{q+\varepsilon -2} \int_{R \cap \{w_k>\lambda\}} w\biggr)\dla .
\]
By Cavalieri's principle with the exponent $q+\varepsilon -1$ and $\dmu = w \dx\dt$, we get
\begin{align*}
\int_{\lambda_0}^\infty \biggl(\lambda^{q+\varepsilon -2} \int_{R \cap \{w_k>\lambda\}} w\biggr) \dla \leq \frac{1}{q+\varepsilon -1} \int_{R \cap \{w_k>\lambda_0\}} w_k^{q+\varepsilon -1} w .
\end{align*}
Consequently,
\begin{align*}
\int_{R^- \cap \{w_k>\lambda_0\}} w_k^{q+\varepsilon -1} w 
\leq 
\frac{C_1 \varepsilon}{q+\varepsilon -1} \int_{R \cap \{w_k>\lambda_0\}} w_k^{q+\varepsilon -1} w
+ \lambda_0^{\varepsilon} \int_{R^- \cap \{w_k>\lambda_0\}} w^q .
\end{align*}
By the boundedness of $w$ and choosing $\varepsilon>0$ to be small enough, we can absorb the integral over $R^- \cap \{w_k>\lambda_0\}$ of the first term to the left-hand side to obtain 
\begin{align*}
&\biggl( 1- \frac{C_1 \varepsilon}{q+\varepsilon -1} \biggr) \int_{R^- \cap \{w_k>\lambda_0\}} w_k^{q+\varepsilon -1} w \\
&\qquad\leq \frac{C_1 \varepsilon}{q+\varepsilon -1} \int_{R^+ \cap \{w_k>\lambda_0\}} w_k^{q+\varepsilon -1} w
+ \lambda_0^{\varepsilon} \int_{R^- \cap \{w_k>\lambda_0\}} w^q .
\end{align*}
Hence, we have
\begin{align*}
\int_{R^- \cap \{w_k>\lambda_0\}} w_k^{q+\varepsilon -1} w \leq c_0 \lambda_0^{\varepsilon} \int_{R^- \cap \{w_k>\lambda_0\}} w^q + c_1 \varepsilon \int_{R^+ \cap \{w_k>\lambda_0\}} w_k^{q+\varepsilon -1} w ,
\end{align*}
where
\[
c_0 = \frac{q+\varepsilon -1}{q+\varepsilon -1- C_1\varepsilon }
\quad\text{and}\quad
c_1 = \frac{C_1}{q+\varepsilon -1- C_1\varepsilon } .
\]
We combine this last estimate with
\begin{align*}
\int_{R^-} w_k^{q+\varepsilon -1} w &= \int_{R^- \cap \{w_k>\lambda_0\}} w_k^{q+\varepsilon -1} w + \int_{R^- \cap \{w_k\leq \lambda_0\}} w_k^{q+\varepsilon -1} w \\
& \leq \int_{R^- \cap \{w_k>\lambda_0\}} w_k^{q+\varepsilon -1} w + \lambda_0^{\varepsilon} \int_{R^- \cap \{w_k\leq \lambda_0\}} w^q
\end{align*}
to obtain
\begin{equation}
\label{iteration-estimate}
\int_{R^-} w_k^{q+\varepsilon -1} w \leq c_0 w_{R^+}^\varepsilon \int_{R^-} w^q + c_1 \varepsilon \int_{R^+} w_k^{q+\varepsilon -1} w .
\end{equation}

Fix $R_0 = Q(x_0,L) \times (t_0 - L^p, t_0 + L^p) \subset \mathbb{R}^{n+1}$.
We cover $R^-_0$ by $M = 2^{n+1}$ rectangles $R^-_{1,j}$ with spatial side length $l_x = L/2^{1/p}$ and time length $l_t = L^p / 2$. This can be done by dividing each spatial edge of $R^-_0$ into two equally long intervals that may overlap each other, and the time interval of $R^-_0$ into two equally long pairwise disjoint intervals.
Observe that the overlap of $R^-_{1,j}$ is bounded by $M/2 = 2^n$.
Then consider $R^+_{1,j}$ and cover it in the same way as before by $M$ rectangles $R^-_{2,j}$ with spatial side length $l_x = L/2^{2/p}$ and time length $l_t = L^p / 2^2$.
At the $i$th step, cover $R^+_{i-1,j}$ by $M$ rectangles $R^-_{i,j}$ with spatial side length $l_x = L/2^{i/p}$ and time length $l_t = L^p / 2^i$ such that their overlap is bounded by $M/2$.
Note that every $R_{i,j}$ is contained in $R_0$.
Then iterating~\eqref{iteration-estimate} we obtain
\begin{align*}
\int_{R^-_0} w_k^{q+\varepsilon -1} w 
&\leq \sum_{j=1}^{M} \int_{R^-_{1,j}} w_k^{q+\varepsilon -1} w \\
&\leq \sum_{j=1}^{M} c_0 w_{R^+_{1,j}}^\varepsilon \int_{R^{-}_{1,j}} w^q + \sum_{j=1}^{M} c_1 \varepsilon \int_{R^+_{1,j}} w_k^{q+\varepsilon -1} w \\
&\leq c_0 \sum_{j=1}^M w_{R^+_{1,j}}^\varepsilon \int_{R^{-}_{1,j}} w^q + c_1 \varepsilon \sum_{j=1}^{M^2} \int_{R^-_{2,j}} w_k^{q+\varepsilon -1} w \\
&\leq c_0 \sum_{j=1}^M w_{R^+_{1,j}}^\varepsilon \int_{R^{-}_{1,j}} w^q\\
&\qquad + c_1 \varepsilon \sum_{j=1}^{M^2} \biggl( c_0 w_{R^+_{2,j}}^\varepsilon \int_{R^{-}_{2,j}} w^q + c_1 \varepsilon \int_{R^+_{2,j}} w_k^{q+\varepsilon -1} w \biggr) \\
&= c_0 \sum_{j=1}^M w_{R^+_{1,j}}^\varepsilon \int_{R^{-}_{1,j}} w^q \\
&\qquad+ c_0 c_1 \varepsilon \sum_{j=1}^{M^2} w_{R^+_{2,j}}^\varepsilon \int_{R^{-}_{2,j}} w^q + (c_1 \varepsilon)^2 \sum_{j=1}^{M^2} \int_{R^+_{2,j}} w_k^{q+\varepsilon -1} w \\
&\leq c_0 \sum_{i=1}^N \biggl( (c_1 \varepsilon)^{i-1} \sum_{j=1}^{M^i} w_{R^+_{i,j}}^\varepsilon \int_{R^{-}_{i,j}} w^q \biggr) + (c_1 \varepsilon)^N \sum_{j=1}^{M^N} \int_{R^+_{N,j}} w_k^{q+\varepsilon -1} w \\
&\leq c_0 \sum_{i=1}^N \biggl( (c_1 \varepsilon)^{i-1} \sum_{j=1}^{M^i} w_{R^+_{i,j}}^\varepsilon \int_{R^{-}_{i,j}} w^q \biggr) + \biggl( c_1 \varepsilon \frac{M}{2} \biggr)^N \int_{R_0} w_k^{q+\varepsilon -1} w \\
&= I + II .
\end{align*}
We observe that $II$ tends to zero if $\varepsilon < 2/(c_1 M) = 1/(c_1 2^n)$ as $N \to \infty$.
For the inner sum of the first term $I$, we have 
\begin{align*}
\sum_{j=1}^{M^i} w_{R^+_{i,j}}^\varepsilon \int_{R^{-}_{i,j}} w^q 
&= \sum_{j=1}^{M^i} \lvert R^+_{i,j} \rvert^{-\varepsilon} w(R^+_{i,j})^\varepsilon \int_{R^{-}_{i,j}} w^q \\
&\leq \sum_{j=1}^{M^i} 2^{\varepsilon+(\frac{n}{p}+1)\varepsilon i} \lvert R_0 \rvert^{-\varepsilon} w(R_0)^\varepsilon \int_{R^{-}_{i,j}} w^q \\
&\leq 2^{\varepsilon+(\frac{n}{p}+1)\varepsilon i} w_{R_{0}}^\varepsilon \biggl(\frac M2\biggr)^i \int_{R_0} w^q .
\end{align*}
Thus, it follows that
\begin{align*}
I \leq c_0 2^{\varepsilon}w_{R_{0}}^\varepsilon \int_{R_0} w^q \sum_{i=1}^N (c_1 \varepsilon)^{i-1} 2^{(\frac{n}{p}+1)\varepsilon i} \biggl(\frac M2\biggr)^i .
\end{align*}
We estimate the sum by
\begin{align*}
\sum_{i=1}^N (c_1 \varepsilon)^{i-1} 2^{(\frac{n}{p}+1)\varepsilon i} \biggl(\frac M2\biggr)^i
&= 2^{(\frac{n}{p}+1)\varepsilon } \frac{M}{2} \sum_{i=0}^{N-1} \biggl( c_1 \varepsilon 2^{(\frac{n}{p}+1)\varepsilon} \frac{M}{2} \biggr)^i \\
&\leq 2^{(\frac{n}{p}+1) \varepsilon } \frac{M}{2} \frac{1}{1-c_1 \varepsilon 2^{(\frac{n}{p}+1) \varepsilon } \frac{M}{2} } \\
&= \frac{2^{(\frac{n}{p}+1) \varepsilon +n }}{1-c_1 \varepsilon 2^{(\frac{n}{p}+1) \varepsilon + n } } = \frac{C}{2^{\varepsilon}c_0} ,
\end{align*}
whenever $\varepsilon$ is small enough, for example
\[
\varepsilon < \frac{1}{c_1 2^{\frac{n}{p}} M } = \frac{1}{c_1 2^{\frac{n}{p}+n+1} } .
\]
Then it holds that
\begin{align*}
\int_{R^-_0} w_k^{q+\varepsilon -1} w \leq C w_{R_{0}}^\varepsilon \int_{R_0} w^q 
\end{align*}
for small enough $\varepsilon$.
The claim follows from the monotone convergence theorem as $k \to \infty$.
\end{proof}

We are ready to prove the parabolic Gehring lemma.

\begin{theorem}
\label{gehring}
Let $1<q<\infty$ and $w$ be a weight.
Assume that there exists a constant $C_1>0$ such that for every parabolic rectangle $R\subset\mathbb{R}^{n+1}$ we have
\begin{equation}
\label{gehringRHI}
\biggl( \dashint_{R^-} w^q \biggr)^\frac{1}{q} \leq C_1 \dashint_{R^+} w .
\end{equation}
Then there exist $\varepsilon=\varepsilon(n,q,C_1)>0$ and $C=C(n,q,C_1)$ such that for every $R\subset\mathbb{R}^{n+1}$ we have
\[
\biggl( \dashint_{R^-} w^{q+\varepsilon} \biggr)^\frac{1}{q+\varepsilon} \leq C \dashint_{R^+} w .
\]
\end{theorem}

\begin{proof}
Our aim is to apply Lemma \ref{selfimprovelemma}.
Let $R_0=R(x_0,t_0,L) = Q(x_0,L) \times (t_0-L^p, t_0+L^p)$ and $\lambda\geq w_{R^+_0}$.
Denote $S^-_0 = R^-_0$.
We construct a parabolic Calder\'on--Zygmund decomposition at level $\lambda$.
We partition $S^-_0$ by dividing each spatial edge into $2$ equally long intervals.
If
\[
\frac{l_t(S_{0}^-)}{\lfloor 2^{p} \rfloor} < \frac{L^p}{2^{p}},
\]
we divide the time interval of $S^-_0$ into $\lfloor 2^{p} \rfloor$ equally long intervals. 
Otherwise, we divide the time interval of $S^-_0$ into $\lceil 2^{p} \rceil$ equally long intervals.
We obtain subrectangles $S^-_1$ of  $S^-_0$ with spatial side length 
$l_x(S^-_1)=l_x(S^-_0)/2 = L / 2$ and time length either 
\[
l_t(S^-_1)=\frac{l_t(S^-_0)}{\lfloor 2^{p} \rfloor} 
=\frac{L^p}{\lfloor 2^{p} \rfloor} 
\quad\text{or}\quad
l_t(S^-_1)=\frac{L^p}{\lceil 2^{p} \rceil}.
\]
For every $S^-_1$, there exists a unique rectangle $R_1$ with spatial side length $l_x = L / 2$ and time length $l_t = 2 L^p / 2^{p}$
such that $R_1$ has the same bottom as $S^-_1$.
Let $S^+_1$ denote the translated $S^-_1$ with the same top as $R_1$.
We select those rectangles $S^-_1$ for which 
\[
\frac{w(S^+_1)}{\lvert S^+_1 \rvert} = \dashint_{S^+_1} w  > \lambda
\]
and denote the obtained collection by $\{ S^-_{1,j} \}_j$.
If 
\[
\frac{w(S^+_1)}{\lvert S^+_1 \rvert} = \dashint_{S^+_1} w \leq \lambda ,
\]
we subdivide $S^-_1$ in the same manner as above
and select all those subrectangles $S^-_2$ for which 
\[
\frac{w(S^+_2)}{\lvert S^+_2 \rvert} = \dashint_{S^+_2} w > \lambda
\]
to obtain family $\{ S^-_{2,j} \}_j$.
We continue this selection process recursively.
At the $i$th step, we partition unselected rectangles $S^-_{i-1}$ by dividing each spatial side into $2$ equally long intervals. 
If 
\begin{equation}
\label{gehring:JNproof_eq1}
\frac{l_t(S_{i-1}^-)}{\lfloor 2^{p} \rfloor} < \frac{L^p}{2^{pi}},
\end{equation}
we divide the time interval of $S^-_{i-1}$ into $\lfloor 2^{p} \rfloor$ equally long intervals. 
If
\begin{equation}
\label{gehring:JNproof_eq2}
\frac{l_t(S_{i-1}^-)}{\lfloor 2^{p} \rfloor} \geq \frac{L^p}{2^{pi}},
\end{equation}
we divide the time interval of $S^-_{i-1}$ into $\lceil 2^{p} \rceil$ equally long intervals.
We obtain subrectangles $S^-_i$. 
For every $S^-_i$, there exists a unique rectangle $R_i$ with spatial side length $l_x = L / 2^{i}$ and time length $l_t = 2 L^p / 2^{pi}$
such that $R_i$ has the same bottom as $S^-_i$.
Let $S^+_i$ denote the translated $S^-_i$ with the same top as $R_i$.
Select those $S^-_i$ for which 
\[
\frac{w(S^+_i)}{\lvert S^+_i \rvert} = \dashint_{S^+_i} w > \lambda
\]
and denote the obtained collection by $\{ S^-_{i,j} \}_j$.
If 
\[
\frac{w(S^+_i)}{\lvert S^+_i \rvert} = \dashint_{S^+_i} w \leq \lambda ,
\]
we continue the selection process in $S^-_i$.
In this manner we obtain a collection $\{S^-_{i,j} \}_{i,j}$ of pairwise disjoint rectangles.

By \eqref{eq:tsidel}, we have 
\[
\frac{1}{2} \frac{L^p}{2^{pi}} \leq l_t(S^-_i) \leq \frac{L^p}{2^{pi}}
\]
for every $S^-_i$.
By using the bounds for the time length of $S^-_i$, we observe that
\begin{align*}
l_t(R_i) - l_t(S^-_i) 
&\leq \frac{2 L^p}{2^{pi}} - \frac{1}{2} \frac{L^p}{2^{pi}} 
= \frac{3}{2} \frac{L^p}{2^{pi}} \\
&\leq \frac{L^p}{2^{p(i-1)}} 
= \frac{2L^p}{2^{p(i-1)}} - \frac{L^p}{2^{p(i-1)}} \\
&\leq l_t(R_{i-1}) - l_t(S^-_{i-1}) .
\end{align*}
This implies that $R_{i} \subset R_{i-1}$
for a fixed rectangle $S^-_{i-1}$ and for every subrectangle $S^-_{i} \subset S^-_{i-1}$.

We have a collection $\{ S^-_{i,j} \}_{i,j}$ of pairwise disjoint rectangles. 
However, the rectangles in the corresponding collection $\{ S^+_{i,j} \}_{i,j}$ may overlap. 
Thus, we replace it by a subfamily $\{ \widetilde{S}^+_{i,j} \}_{i,j}$ of pairwise disjoint rectangles, which is constructed in the following way.
At the first step, choose $\{ S^+_{1,j} \}_{j}$ and denote it by $\{ \widetilde{S}^+_{1,j} \}_j$. 
Then consider the collection $\{ S^+_{2,j} \}_{j}$ where each $S^+_{2,j}$ either intersects some $\widetilde{S}^+_{1,j}$ or does not intersect any $\widetilde{S}^+_{1,j}$. 
Select the rectangles $S^+_{2,j}$ that do not intersect any $\widetilde{S}^+_{1,j}$, and denote the obtained collection by $\{ \widetilde{S}^+_{2,j} \}_j$.
At the $i$th step, choose those $S^+_{i,j}$ that do not intersect any previously selected $\widetilde{S}^+_{i',j}$, $i' < i$.
Hence, we obtain a collection $\{ \widetilde{S}^+_{i,j} \}_{i,j}$ of pairwise disjoint rectangles.
Observe that for every $S^+_{i,j}$ there exists $\widetilde{S}^+_{i',j}$ with $i' < i$ such that
\begin{equation}
\label{gehring:plussubset}
\text{pr}_x(S^+_{i,j}) \subset \text{pr}_x(\widetilde{S}^+_{i',j}) \quad \text{and} \quad \text{pr}_t(S^+_{i,j}) \subset 3 \text{pr}_t(\widetilde{S}^+_{i',j}) .
\end{equation}
Here pr$_x$ denotes the projection to $\mathbb R^n$ and pr$_t$ denotes the projection to the time axis.

Rename $\{ S^-_{i,j} \}_{i,j}$ and $\{ \widetilde{S}^+_{i,j} \}_{i,j}$ as $\{ S^-_{i} \}_{i}$ and $\{ \widetilde{S}^+_{j} \}_j$, respectively.
Note that $S^-_i$ is spatially contained in $S^+_i$, that is, $\text{pr}_x S^-_i\subset \text{pr}_x S^+_i$.
In the time direction, we have
\begin{equation}
\label{gehring:minusplussubset}
\text{pr}_t(S^-_i) \subset \text{pr}_t(R_i) 
\subset 7 \text{pr}_t(S^+_i) ,
\end{equation}
since
\[
( 7 + 1 ) \frac{l_t(S^+_i)}{2} \geq 8 \frac{L^p}{2^{pi+2}} = \frac{2L^p}{2^{pi}} = l_t(R_i) .
\]
Therefore, by~\eqref{gehring:plussubset} and~\eqref{gehring:minusplussubset}, it holds that
\begin{equation}
\label{gehring:subsetcubes}
\sum_i \lvert S^-_i \rvert = \Big\lvert \bigcup_i S^-_i \Big\rvert \leq c_1 \sum_j \lvert \widetilde{S}^+_j \rvert 
\quad\text{with}\quad
c_1 = 21.
\end{equation}

If $(x,t) \in R^-_0 \setminus \bigcup_i S^-_i$, then there exists a sequence $\{S^-_l\}_{l\in\mathbb N}$ of subrectangles containing $(x,t)$ such that 
\[
\frac{w(S^+_l)}{\lvert S^+_l \rvert} = \dashint_{S^+_l} w \leq \lambda
\]
and $\lvert S^-_l \rvert \to 0$ as $l \to \infty$.
The Lebesgue differentiation theorem~\cite[Lemma~2.3]{KinnunenMyyryYang2022}
implies that $w(x,t) \leq \lambda$
for almost every $(x,t) \in R^-_0 \setminus \bigcup_i S^-_i$.
It follows that
\begin{equation}
\label{gehring:levelsetsubset}
R^-_0 \cap \{ w > \lambda \} \subset \bigcup_i S^-_i
\end{equation}
up to a set of measure zero.

Consider $S^-_i$ and denote its parent by $S^-_{i-1}$, that is, $S^-_{i}$ was obtained by subdividing the previous $S^-_{i-1}$ for which $w_{S^+_{i-1}} \leq\lambda$.
We move the corresponding $R^+_i$ forward in time until the shifted rectangle is contained in $S^+_{i-1}$.
The time distance between the bottom of $R^+_i$ and the bottom of $S^+_{i-1}$ is bounded above by $2^{p+1} l_t(R^+_i)$.
The assumption \eqref{gehringRHI} with H\"older's inequality implies that $w(R^-) \leq C_1 w(R^+) $ for every parabolic rectangle $R$.
Thus, we can apply the proof of Lemma~\ref{lemma:timemove}~$(ii)$ with $\theta = 2^{p+1}$ to obtain
\begin{equation}
\label{gehring:shifting}
w(R^+_i) \leq 4 \max\{1, C_1^{1+ 2^{2p+1} }\} w(S^+_{i-1})
\end{equation}
for every $i\in\mathbb{N}$.

By using \eqref{gehring:levelsetsubset}, \eqref{gehringRHI}, \eqref{gehring:shifting} and \eqref{gehring:subsetcubes}, we obtain
\begin{equation}
\label{eq:gehring_estimate}
\begin{split}
\int_{R^-_0 \cap \{ w>\lambda \}} w^q &\leq \sum_i \int_{S^-_i} w^q \leq \sum_i \int_{R^-_i} w^q \leq C_1^q \sum_i \lvert R^-_i \rvert \biggl( \dashint_{R^+_i} w \biggr)^q \\
&\leq C_1^q 4^q \max\{1, C_1^{q(1+ 2^{2p+1}) }\} \sum_i \lvert R^-_i \rvert \biggl( \frac{\lvert S^+_{i-1} \rvert}{\lvert R^+_i \rvert} \dashint_{S^+_{i-1}} w \biggr)^q \\
&\leq c_2 \lambda^q \sum_i \lvert R^-_i \rvert \leq 2 c_2 \lambda^q \sum_i \lvert S^-_i \rvert \\
&\leq 2 c_1 c_2 \lambda^q \sum_j \lvert \widetilde{S}^+_j \rvert ,
\end{split}
\end{equation}
where $c_2 = 2^{q(2+n+p)} C_1^q \max\{1, C_1^{q(1+ 2^{2p+1}) }\} $. 
We have
\begin{align*}
\lvert \widetilde{S}^+_j \rvert &\leq \frac{1}{\lambda} \int_{\widetilde{S}^+_j} w = \frac{1}{\lambda} \int_{\widetilde{S}^+_j \cap \{w>\lambda/2\}} w + \frac{1}{\lambda} \int_{\widetilde{S}^+_j \cap \{w\leq\lambda/2\}} w \\
&\leq \frac{1}{\lambda} \int_{\widetilde{S}^+_j \cap \{w>\lambda/2\}} w + \frac{1}{\lambda} \int_{\widetilde{S}^+_j \cap \{w\leq\lambda/2\}} \frac{\lambda}{2} \\
&\leq \frac{1}{\lambda} \int_{\widetilde{S}^+_j \cap \{w>\lambda/2\}} w + \frac{1}{2} \lvert \widetilde{S}^+_j \rvert ,
\end{align*}
and thus
\[
\lvert \widetilde{S}^+_j \rvert \leq \frac{2}{\lambda} \int_{\widetilde{S}^+_j \cap \{w>\lambda/2\}} w .
\]
From this and \eqref{eq:gehring_estimate},
it follows that
\begin{align*}
\int_{R^-_0 \cap \{ w>\lambda \}} w^q &\leq 2 c_1 c_2 \lambda^q \sum_j \lvert \widetilde{S}^+_j \rvert \leq 4 c_1 c_2 \lambda^{q-1} \sum_j \int_{\widetilde{S}^+_j \cap \{w>\lambda/2\}} w \\
&= 4 c_1 c_2 \lambda^{q-1} \int_{\bigcup_j \widetilde{S}^+_j \cap \{w>\lambda/2\}} w
\leq 4 c_1 c_2 \lambda^{q-1} \int_{R_0 \cap \{w>\lambda/2\}} w ,
\end{align*}
since $\widetilde{S}^+_j$ are pairwise disjoint.
On the other hand, we have
\begin{align*}
\int_{R^-_0 \cap \{\lambda\geq w>\lambda/2\}} w^q = \int_{R^-_0 \cap \{\lambda\geq w>\lambda/2\}} w^{q-1} w \leq \lambda^{q-1} \int_{R_0 \cap \{w>\lambda/2\}} w .
\end{align*}
Combining the two previous estimates, we get
\begin{align*}
\int_{R^-_0 \cap \{w>\lambda/2\}} w^q 
&= \int_{R^-_0 \cap \{w>\lambda\}} w^q + \int_{R^-_0 \cap \{\lambda\geq w>\lambda/2\}} w^q\\
& \leq c_3 \biggl( \frac{\lambda}{2} \biggr)^{q-1} \int_{R_0 \cap \{w>\lambda/2\}} w
\end{align*}
for $\lambda\geq w_{R^+_0}$, 
where $c_3 = 2^{q-1} (4 c_1 c_2 + 1)$.
Since this holds for any parabolic rectangle $R_0$, we may apply Lemma~\ref{selfimprovelemma} 
which states that there exist $\varepsilon>0$ and $C>1$ such that
\begin{align*}
\int_{R^-} w^{q+\varepsilon} 
&\leq C \biggl( \dashint_{R} w \biggr)^\varepsilon \int_{R} w^q \\
&=\frac{C}{2^{\varepsilon}} \biggl( \dashint_{R^-} w+\dashint_{R^+} w \biggr)^\varepsilon \int_{R} w^q \\
&\le\frac{C}{2^{\varepsilon}} \biggl(C_1\dashint_{R^+} w+\dashint_{R^+} w \biggr)^\varepsilon \int_{R} w^q \\
&\le\frac{C(C_1+1)^{\varepsilon}}{2^{\varepsilon}}\biggl(\dashint_{R^+} w\biggr)^\varepsilon \int_{R} w^q 
\end{align*}
for every parabolic rectangle $R \subset \mathbb{R}^{n+1}$.
Here we also apply \eqref{gehringRHI}.
We estimate the second integral on the right-hand side similarly to get
\begin{align*}
\int_{R} w^q &= \int_{R^-} w^q + \int_{R^+} w^q \\
&\leq C_1^q \lvert R^- \rvert \biggl( \dashint_{R^+} w \biggr)^q + C_1^q \lvert R^+ \rvert \biggl( \dashint_{R^{++}} w \biggr)^q \\
&\leq C_1^{2q} \lvert R^- \rvert \biggl( \dashint_{R^{++}} w \biggr)^q + C_1^q \lvert R^- \rvert \biggl( \dashint_{R^{++}} w \biggr)^q \\
&= C_2 \lvert R^- \rvert \biggl( \dashint_{R^{++}} w \biggr)^q ,
\end{align*}
where $R^{++} = R^+ + (0, l_t(R^+))$ and $C_2 =C_1^{2q}+C_1^{q}$.
Therefore, we have
\begin{align*}
\int_{R^-} w^{q+\varepsilon} 
&\leq \frac{C(C_1+1)^{\varepsilon}}{2^{\varepsilon}} \biggl( \dashint_{R^+} w \biggr)^\varepsilon \int_{R} w^q \\
&\leq \frac{C(C_1+1)^{\varepsilon}}{2^{\varepsilon}} C_1^\varepsilon C_2 \lvert R^- \rvert \biggl( \dashint_{R^{++}} w \biggr)^\varepsilon  \biggl( \dashint_{R^{++}} w \biggr)^q \\
&= C_3^{q+\varepsilon} \lvert R^- \rvert \biggl( \dashint_{R^{++}} w \biggr)^{q+\varepsilon} ,
\end{align*}
where $C_3^{q+\varepsilon} = 2^{-\varepsilon}C(C_1+1)^{\varepsilon}C_1^\varepsilon C_2$.
We conclude that
\[
\biggl( \dashint_{R^-} w^{q+\varepsilon} \biggr)^\frac{1}{q+\varepsilon} \leq C_3 \dashint_{R^{++}} w 
\]
for every parabolic rectangle $R \subset \mathbb{R}^{n+1}$.
It is left to replace $R^{++}$ by $R^+$ in the estimate above. This is done by the following argument.

Fix $R_0 = Q(x_0,L) \times (t_0 - L^p, t_0 + L^p) \subset \mathbb{R}^{n+1}$.
We cover $R^-_0$ by $M = 2^{n+1}$ rectangles $R^-_{i}$ with spatial side length $l_x = L/2^{1/p}$ and time length $l_t = L^p / 2$. This can be done by dividing each spatial edge of $R^-_0$ into two equally long intervals that may overlap each other, and the time interval of $R^-_0$ into two equally long pairwise disjoint intervals.
Observe that every $R^{++}_i$ is contained in $R^+_0$ and
the overlap of $R^{++}_{i}$ is bounded by $M/2 = 2^n$.
Then it holds that
\begin{align*}
\biggl( \dashint_{R^-_0} w^{q+\varepsilon} \biggr)^\frac{1}{q+\varepsilon} &\leq \biggl( \sum_i \frac{\lvert R^-_i \rvert}{\lvert R^-_0 \rvert} \dashint_{R^-_i} w^{q+\varepsilon} \biggr)^\frac{1}{q+\varepsilon} \\&\leq 2^{-(\frac{n}{p}+1)/(q+\varepsilon)}  \sum_i \biggl( \dashint_{R^-_i} w^{q+\varepsilon} \biggr)^\frac{1}{q+\varepsilon} \\
&\leq 2^{-(\frac{n}{p}+1)/(q+\varepsilon)} C_3 \sum_i \dashint_{R^{++}_i} w \\
&= 2^{-(\frac{n}{p}+1)/(q+\varepsilon)} C_3 \sum_i \frac{\lvert R^+_0 \rvert}{\lvert R^{++}_i \rvert} \frac{1}{\lvert R^+_0 \rvert} \int_{R^{++}_i} w \\
&\leq 2^{-(\frac{n}{p}+1)/(q+\varepsilon)} C_3 2^{\frac{n}{p}+1} \frac{M}{2} \dashint_{R^{+}_0} w \\
&= 2^{n+(\frac{n}{p}+1)(1-1/(q+\varepsilon))} C_3 \dashint_{R^{+}_0} w .
\end{align*}
This completes the proof.
\end{proof}

In addition to the self-improvement of the exponent on the left-hand side of the parabolic reverse H\"older inequality,
we observe that the exponent on the right-hand side can be replaced by any smaller positive exponent. 
For the elliptic case, for example, see \cite[Lemma 3.38]{Heinonen_et_al}.

\begin{theorem}
\label{thm:RHI-(1,s)}
Let $1<q<\infty$ and $w$ be a weight.
Assume that there exists a constant $C_1>0$ such that for every parabolic rectangle $R\subset\mathbb{R}^{n+1}$ we have
\begin{equation}
\label{thm:RHIforRHS}
\biggl( \dashint_{R^-} w^q \biggr)^\frac{1}{q} \leq C_1 \dashint_{R^+} w .
\end{equation}
Then for every $0<s<1$ there exists a constant $C=C(n,p,q,s,C_1)$ such that for every $R\subset\mathbb{R}^{n+1}$ we have
\[
\biggl( \dashint_{R^-} w^{q} \biggr)^\frac{1}{q} \leq C \biggl( \dashint_{R^+} w^s \biggr)^\frac{1}{s} .
\]
\end{theorem}

\begin{proof}
Let $R \subset \mathbb R^{n+1}$ be a parabolic rectangle.
Fix $0<s<1$. Let $\theta=s(q-1)/(q-s)$, that is,
\[
1 = \frac{\theta}{s} + \frac{1-\theta}{q} .
\]
We apply H\"older's inequality, Young's inequality
\[
ab \leq \varepsilon a^r + \varepsilon^{-\frac{1}{r-1}} b^\frac{r}{r-1}
\]
with $r=1/(1-\theta)$ and~\eqref{thm:RHIforRHS} to get
\begin{align*}
\dashint_{R^-} w 
&= \dashint_{R^-} w^\theta w^{1-\theta} 
\leq \biggl( \dashint_{R^-} w^s \biggr)^\frac{\theta}{s} \biggl( \dashint_{R^-} w^q \biggr)^\frac{1-\theta}{q}\\
&\leq \varepsilon^{1-\frac{1}{\theta}} \biggl( \dashint_{R^-} w^s \biggr)^\frac{1}{s} + \varepsilon \biggl( \dashint_{R^-} w^q \biggr)^\frac{1}{q} \\
&\leq \varepsilon^{1-\frac{1}{\theta}} \biggl( \dashint_{R^-} w^s \biggr)^\frac{1}{s} + C_1 \varepsilon \dashint_{R^+} w .
\end{align*}
Hence, we have
\begin{equation}
\label{iteration-(1,s)-RHI}
\int_{R^-} w \leq \varepsilon^{1-\frac{1}{\theta}} \lvert R^- \rvert \biggl( \dashint_{R^-} w^s \biggr)^\frac{1}{s} + C_1 \varepsilon \int_{R^+} w .
\end{equation}

Fix $R_0 = Q(x_0,L) \times (t_0 - L^p, t_0 + L^p) \subset \mathbb{R}^{n+1}$.
We cover $R^-_0$ by $M = 2^{n+1}$ rectangles $R^-_{1,j}$ with spatial side length $l_x = L/2^{1/p}$ and time length $l_t = L^p / 2$. This can be done by dividing each spatial edge of $R^-_0$ into two equally long intervals that may overlap each other, and the time interval of $R^-_0$ into two equally long pairwise disjoint intervals.
Observe that the overlap of $R^-_{1,j}$ is bounded by $M/2 = 2^n$.
Then consider $R^+_{1,j}$ and cover it in the same way as before by $M$ rectangles $R^-_{2,j}$ with spatial side length $l_x = L/2^{2/p}$ and time length $l_t = L^p / 2^2$.
At the $i$th step, cover $R^+_{i-1,j}$ by $M$ rectangles $R^-_{i,j}$ with spatial side length $l_x = L/2^{i/p}$ and time length $l_t = L^p / 2^i$ such that their overlap is bounded by $M/2$.
Note that every $R_{i,j}$ is contained in $R_0$.
Then iterating~\eqref{iteration-(1,s)-RHI} we obtain
\begin{align*}
\int_{R^-_0} w &\leq \sum_{j=1}^{M} \int_{R^-_{1,j}} w  \leq \sum_{j=1}^{M} \varepsilon^{1-\frac{1}{\theta}} \lvert R^-_{1,j} \rvert \biggl( \dashint_{R^-_{1,j}} w^s \biggr)^\frac{1}{s} + \sum_{j=1}^{M} C_1 \varepsilon \int_{R^+_{1,j}} w \\
&\leq \varepsilon^{1-\frac{1}{\theta}} \sum_{j=1}^M \lvert R^-_{1,j} \rvert \biggl( \dashint_{R^-_{1,j}} w^s \biggr)^\frac{1}{s} + C_1 \varepsilon \sum_{j=1}^{M^2} \int_{R^-_{2,j}} w \\
&\leq \varepsilon^{1-\frac{1}{\theta}} \sum_{j=1}^M \lvert R^-_{1,j} \rvert \biggl( \dashint_{R^-_{1,j}} w^s \biggr)^\frac{1}{s}\\
&\qquad + C_1 \varepsilon \sum_{j=1}^{M^2} \biggl( \varepsilon^{1-\frac{1}{\theta}} \lvert R^-_{2,j} \rvert \biggl( \dashint_{R^-_{2,j}} w^s \biggr)^\frac{1}{s} + C_1 \varepsilon \int_{R^+_{2,j}} w \biggr) \\
&= \varepsilon^{1-\frac{1}{\theta}} \sum_{j=1}^M \lvert R^-_{1,j} \rvert \biggl( \dashint_{R^-_{1,j}} w^s \biggr)^\frac{1}{s} \\
&\qquad+ \varepsilon^{1-\frac{1}{\theta}} C_1 \varepsilon \sum_{j=1}^{M^2} \lvert R^-_{2,j} \rvert \biggl( \dashint_{R^-_{2,j}} w^s \biggr)^\frac{1}{s} + (C_1 \varepsilon)^2 \sum_{j=1}^{M^2} \int_{R^+_{2,j}} w \\
&\leq \varepsilon^{1-\frac{1}{\theta}} \sum_{i=1}^N \biggl( (C_1 \varepsilon)^{i-1} \sum_{j=1}^{M^i} \lvert R^-_{i,j} \rvert \biggl( \dashint_{R^-_{i,j}} w^s \biggr)^\frac{1}{s} \biggr) + (C_1 \varepsilon)^N \sum_{j=1}^{M^N} \int_{R^+_{N,j}} w \\
&\leq \varepsilon^{1-\frac{1}{\theta}} \sum_{i=1}^N \biggl( (C_1 \varepsilon)^{i-1} \sum_{j=1}^{M^i} \lvert R^-_{i,j} \rvert \biggl( \dashint_{R^-_{i,j}} w^s \biggr)^\frac{1}{s} \biggr) + 
\biggl( C_1 \varepsilon \frac{M}{2} \biggr)^N 
\int_{R_0} w \\
&= I + II .
\end{align*}
We observe that $II$ tends to zero if $\varepsilon < 2/(C_1 M) = 1/(C_1 2^n)$ as $N \to \infty$.
For the inner sum of the first term $I$, we have 
\begin{align*}
\sum_{j=1}^{M^i} \lvert R^-_{i,j} \rvert \biggl( \dashint_{R^-_{i,j}} w^s \biggr)^\frac{1}{s}
&= \sum_{j=1}^{M^i} \lvert R^-_{i,j} \rvert^{1-\frac{1}{s}} \biggl( \int_{R^-_{i,j}} w^s \biggr)^\frac{1}{s}\\
&\leq \sum_{j=1}^{M^i} 2^{(\frac{n}{p}+1)(\frac{1}{s}-1) i} \lvert R^-_0 \rvert^{1-\frac{1}{s}} \biggl( \int_{R^-_{i,j}} w^s \biggr)^\frac{1}{s} \\
&\leq 2^{(\frac{n}{p}+1)(\frac{1}{s}-1) i +\frac{1}{s}} \biggl( \frac{M}{2} \biggr)^i \lvert R^-_0 \rvert \biggl( \dashint_{R_{0}} w^s \biggr)^\frac{1}{s} .
\end{align*}
Thus, it follows that
\begin{align*}
I \leq \varepsilon^{1-\frac{1}{\theta}} 2^{\frac{1}{s}} \lvert R^-_0 \rvert \biggl( \dashint_{R_{0}} w^s \biggr)^\frac{1}{s} \sum_{i=1}^N (C_1 \varepsilon)^{i-1} 2^{(\frac{n}{p}+1)(\frac{1}{s}-1) i} \biggl( \frac{M}{2} \biggr)^i .
\end{align*}
We estimate the sum by
\begin{align*}
\sum_{i=1}^N (C_1 \varepsilon)^{i-1} 2^{(\frac{n}{p}+1)(\frac{1}{s}-1) i} \biggl( \frac{M}{2} \biggr)^i
&= 2^{(\frac{n}{p}+1)(\frac{1}{s}-1) +n} \sum_{i=0}^{N-1} \bigl( C_1 \varepsilon 2^{(\frac{n}{p}+1)(\frac{1}{s}-1) +n} \bigr)^i \\
&= \frac{2^{(\frac{n}{p}+1)(\frac{1}{s}-1) +n }}{1-C_1 \varepsilon 2^{(\frac{n}{p}+1)(\frac{1}{s}-1) + n } } ,
\end{align*}
whenever
$\varepsilon < 1/ (C_1 2^{(\frac{n}{p}+1)(\frac{1}{s}-1) + n })$.
Then it holds that
\begin{align*}
\int_{R^-_0} w \leq \varepsilon^{1-\frac{1}{\theta}} \frac{ 2^{(\frac{n}{p}+1)(\frac{1}{s}-1) +n +\frac{1}{s} }}{1-C_1 \varepsilon 2^{(\frac{n}{p}+1)(\frac{1}{s}-1) + n } }  \lvert R^-_0 \rvert \biggl( \dashint_{R_{0}} w^s \biggr)^\frac{1}{s}
\end{align*}
for 
\[
0< \varepsilon < \min\biggl\{ \frac{1}{C_1 2^n} , \frac{1}{C_1 2^{(\frac{n}{p}+1)(\frac{1}{s}-1) + n } } \biggr\}
= \frac{1}{C_1 2^{(\frac{n}{p}+1)(\frac{1}{s}-1) + n } } .
\]
Choose $\varepsilon = 1/ (C_1 2^{(\frac{n}{p}+1)(\frac{1}{s}-1) + n +1}) $.
By the arbitrariness of $R_0$ and~\eqref{thm:RHIforRHS}, we conclude that
\begin{equation}
\label{eq:(q,s)-RHI}
\biggl( \dashint_{R^{--}} w^{q} \biggr)^\frac{1}{q} \leq C_1 \dashint_{R^-} w \leq C \biggl( \dashint_{R} w^s \biggr)^\frac{1}{s}
\end{equation}
for every parabolic rectangle $R \subset \mathbb{R}^{n+1}$,
where $R^{--} = R^- - l_t(R^-)$ and
\[
C = C_1 \varepsilon^{\frac{q(s-1)}{s(q-1)}} \frac{  2^{(\frac{n}{p}+1)(\frac{1}{s}-1) +n +\frac{1}{s} }}{1-C_1 \varepsilon 2^{(\frac{n}{p}+1)(\frac{1}{s}-1) + n } } .
\]

Fix $R_0 = Q(x_0,L) \times (t_0 - L^p, t_0 + L^p) \subset \mathbb{R}^{n+1}$.
We cover $Q(x_0,L) \times (t_0 - L^p, t_0 - L^p/2)$ by $2^n$ rectangles $R^{-}_{1,i}$ with spatial side length $l_x = L/2^{1/p}$ and time length $l_t = L^p / 2$ by dividing each edge of $Q(x_0,L)$ into two equally long intervals that may overlap each other. 
Denote $R^{--}_{2,i} = R^+_{1,i}$.
Observe that the union of $R^{--}_{2,i}$ covers $Q(x_0,L) \times (t_0 - L^p/2, t_0)$.
Moreover, note that every $R_{2,i}$ is contained in $R^+_0$.
Then by~\eqref{eq:(q,s)-RHI}, we have
\begin{align*}
\biggl( \dashint_{R^-_0} w^q \biggr)^\frac{1}{q} &\leq \biggl( \frac{\lvert R^-_{1,i} \rvert}{\lvert R^-_0 \rvert} \sum_i \biggl( \dashint_{R^-_{1,i}} w^q +  \dashint_{R^{--}_{2,i}} w^q \biggr) \biggr)^\frac{1}{q} \\
&\leq 2^{-(\frac{n}{p}+1)/q} \sum_i \biggl( \biggl( \dashint_{R^-_{1,i}} w^q \biggr)^\frac{1}{q} + \biggl( \dashint_{R^{--}_{2,i}} w^q \biggr)^\frac{1}{q} \biggr) \\
&\leq 2^{-(\frac{n}{p}+1)/q} (C_1+1) \sum_i \biggl( \dashint_{R^{--}_{2,i}} w^q \biggr)^\frac{1}{q} \\
&\leq 2^{-(\frac{n}{p}+1)/q} (C_1+1) C \sum_i \biggl( \dashint_{R_{2,i}} w^s \biggr)^\frac{1}{s} \\
&\leq 2^{-(\frac{n}{p}+1)/q} (C_1+1) C 2^n 2^{\frac{n}{p}\frac{1}{s}} \biggl( \dashint_{R^+_0} w^s \biggr)^\frac{1}{s} .
\end{align*}
This completes the proof.
\end{proof}

\section{Connection to parabolic Muckenhoupt weights}\label{section:muckenhoupt}

In this section, we show that the parabolic reverse H\"older inequality together with the following parabolic doubling condition implies the parabolic Muckenhoupt condition.
We recall the definition of parabolic Muckenhoupt classes $A^+_q$.

\begin{definition}
Let $1<q<\infty$ and $0<\gamma<1$.
A weight $w$ belongs to the parabolic Muckenhoupt class $A^+_q(\gamma)$ if
\[
[w]_{A^+_q(\gamma)} = 
\sup_{R \subset \mathbb R^{n+1}}
\biggl( \dashint_{R^-(\gamma)} w \biggr) \biggl( \dashint_{R^+(\gamma)} w^{\frac{1}{1-q}} \biggr)^{q-1} < \infty ,
\]
where the supremum is taken over all parabolic rectangles $R \subset \mathbb R^{n+1}$.
If the condition above holds with the time axis reversed, then $w \in A^-_q(\gamma)$.
\end{definition}

We say that a measure is forward in time parabolic doubling if
\begin{equation}
\label{eq:pardoubling}
w(R^-(\gamma)) \leq c_d w\Bigl(\frac{1}{2}R^+(\gamma)\Bigr)
\end{equation}
for every parabolic rectangle 
$R=R(x,t,L)\subset\mathbb{R}^{n+1}$,
where $c_d>0$ is the parabolic doubling constant.
Here 
\[
\frac{1}{2}R^+(\gamma)=Q\Bigl(x,\frac L2\Bigr)\times\biggl(t+\frac{1+\gamma}{2}L^p-\frac{1-\gamma}{2}\frac{L^p}{2^p},t+\frac{1+\gamma}{2}L^p+\frac{1-\gamma}{2}\frac{L^p}{2^p}\biggr).
\] 
Note that $\frac{1}{2}R^+(\gamma)$ has the same center as $R^+(\gamma)$, $\frac{1}{2}R^+(\gamma)\subset R^+(\gamma)$ and
$2^{n+p}\lvert\frac{1}{2}R^+(\gamma)\rvert=\lvert R^+(\gamma)\rvert$.
Moreover, there exists a parabolic rectangle $S$ such that $S^+(\gamma)=\frac{1}{2}R^+(\gamma)$.

\begin{lemma}
\label{lemma:measureweight-estimate}
Let $w$
be a weight 
satisfying 
\eqref{eq:pardoubling}
with $0<\gamma<1$.
Assume that there exist $0<\alpha<1$ and $0<\beta<2^{n+p-1}/c_d^2$ such that for every parabolic rectangle $R$ and every measurable set $E \subset R^-(\gamma)$ for which $\lvert E \rvert < \alpha \lvert R^-(\gamma) \rvert$ we have $w(E) < \beta w(R^+(\gamma))$.
Then there exist $\tau = \tau(p,\gamma) \geq 1$, $\rho=\rho(\alpha,\beta)<1$ and $c=c(n,p,\gamma,\alpha,\beta)$ such that for every parabolic rectangle $R=R(x,t,L) \subset \mathbb{R}^{n+1}$ and $\lambda \geq (w_{U^-})^{-1}$ we have
\[
\lvert R^{+}(\gamma) \cap \{ w^{-1} > \lambda \} \rvert \leq c \lambda w( R^{\tau} \cap \{ w^{-1} > \rho \lambda \}  ) ,
\]
where
\[
U^- = R^+(\gamma) - (0, \tau (1+\gamma) L^p)
\]
and
\[ 
R^{\tau} = Q(x,L) \times (t+\gamma L^p - \tau(1+\gamma)L^p , t+L^p) .
\]
Note that $U^- = R^-(\gamma)$ and $R^\tau = R$ for $\tau=1$.
\end{lemma}

\begin{proof}
Let $R_0=R(x_0,t_0,L) = Q(x_0,L) \times (t_0-L^p, t_0+L^p)$.
Denote $f=w^{-1}$ and $\dmu = w \dx \dt$.
Let $\tau\geq 1$ to be chosen later.
Denote $S^+_0 = R^+_0(\gamma)$. The time length of $S^+_0$ is $l_t(S^+_0) = (1-\gamma) L^p$.
We construct a parabolic Calder\'on--Zygmund decomposition at level $\lambda$.
We partition $S^+_0$ by dividing each spatial edge into $2$ equally long intervals. If
\[
\frac{l_t(S_{0}^+)}{\lceil 2^{p} \rceil} > \frac{(1-\gamma) L^p}{2^{p}},
\]
we divide the time interval of $S^+_0$ into $\lceil 2^{p} \rceil$ equally long intervals. Otherwise, we divide the time interval of $S^+_0$ into $\lfloor 2^{p} \rfloor$ equally long intervals.
We obtain subrectangles $S^+_1$ of $S^+_0$ with spatial side length $L_1 = l_x(S^+_1) = l_x(S^+_0)/2 = L / 2$
and time length either
\[
l_t(S^+_1) = \frac{l_t(S^+_0)}{\lceil 2^{p} \rceil} = \frac{(1-\gamma) L^p}{\lceil 2^{p} \rceil} \quad \text{or} \quad l_t(S^+_1) = \frac{(1-\gamma) L^p}{\lfloor 2^{p} \rfloor} .
\]
For every $S^+_1$, there exists a unique rectangle $R_1$ with spatial side length $L_1 = L / 2$ 
and time length $2L_1^p = 2 L^p / 2^{p}$
such that $R_1$ has the same top as $S^+_1$.
Let $U_1^- = R_1^+(\gamma) - (0, \tau (1+\gamma) L_1^p ) $.
We select those rectangles $S^+_1$ for which
\[
\frac{\lvert U^-_1 \rvert}{w(U^-_1)} = \dashint_{U^-_1} f \dmu > \lambda
\]
and denote the obtained collection by $\{ S^+_{1,j} \}_j$.
If
\[
\frac{\lvert U^-_1 \rvert}{w(U^-_1)} = \dashint_{U^-_1} f \dmu \leq \lambda ,
\]
we subdivide $S^+_1$ in the same manner as above
and select all those subrectangles $S^+_2$ for which
\[
\frac{\lvert U^-_2 \rvert}{w(U^-_2)} = \dashint_{U^-_2} f \dmu > \lambda
\]
to obtain family $\{ S^+_{2,j} \}_j$.
We continue this selection process recursively.
At the $i$th step, we partition unselected rectangles $S^+_{i-1}$ by dividing each spatial side into $2$ equally long intervals. If
\begin{equation}
\label{eq:Aqproof_eq1}
\frac{l_t(S_{i-1}^+)}{\lceil 2^{p} \rceil} > \frac{(1-\gamma) L^p}{2^{pi}},
\end{equation}
we divide the time interval of $S^+_{i-1}$ into $\lceil 2^{p} \rceil$ equally long intervals. 
Otherwise, if
\begin{equation}
\label{eq:Aqproof_eq2}
\frac{l_t(S_{i-1}^+)}{\lceil 2^{p} \rceil} \leq \frac{(1-\gamma) L^p}{2^{pi}},
\end{equation}
we divide the time interval of $S^+_{i-1}$ into $\lfloor 2^{p} \rfloor$ equally long intervals.
We obtain subrectangles $S^+_i$. For every $S^+_i$, there exists a unique rectangle $R_i$ with spatial side length $L_i = L / 2^{i}$
and time length $2L_i^p = 2 L^p / 2^{pi}$
such that $R_i$ has the same top as $S^+_i$.
Let $U_i^- = R_i^+(\gamma) - (0, \tau (1+\gamma) L_i^p )$.
Select those $S^+_i$ for which 
\[
\frac{\lvert U^-_i \rvert}{w(U^-_i)} = \dashint_{U^-_i} f \dmu > \lambda
\]
and denote the obtained collection by $\{ S^+_{i,j} \}_j$.
If
\[
\frac{\lvert U^-_i \rvert}{w(U^-_i)} = \dashint_{U^-_i} f \dmu \leq \lambda
\]
we continue the selection process in $S^+_i$.
In this manner we obtain a collection $\{S^+_{i,j} \}_{i,j}$ of pairwise disjoint rectangles.

Observe that if \eqref{eq:Aqproof_eq1} holds, then we have
\[
l_t(S_i^+) = \frac{l_t(S^+_{i-1})}{\lceil 2^{p} \rceil} \geq \frac{(1-\gamma) L^p}{2^{pi}}.
\]
On the other hand, if \eqref{eq:Aqproof_eq2} holds, then
\[
l_t(S_i^+) = \frac{l_t(S^+_{i-1})}{\lfloor 2^{p} \rfloor} \geq \frac{l_t(S^+_{i-1})}{2^{p}} \geq \dots \geq \frac{(1-\gamma)  L^p}{2^{pi}} .
\]
This gives a lower bound 
\[
l_t(S_i^+) \geq \frac{ (1-\gamma) L^p}{2^{pi}}
\]
for every $S_i^+$.

Suppose that \eqref{eq:Aqproof_eq2} is satisfied at the $i$th step.
Then we have an upper bound for the time length of $S_i^+$, since
\begin{align*}
l_t(S^+_i) = \frac{l_t(S_{i-1}^+)}{\lfloor 2^{p} \rfloor} \leq \frac{\lceil 2^{p} \rceil}{\lfloor 2^{p} \rfloor} \frac{(1-\gamma) L^p}{2^{pi}} \leq \biggl( 1+ \frac{1}{\lfloor 2^{p} \rfloor} \biggr) \frac{(1-\gamma) L^p}{2^{pi}} .
\end{align*}
On the other hand, if \eqref{eq:Aqproof_eq1} is satisfied, then
\[
l_t(S^+_i) = \frac{l_t(S_{i-1}^+)}{\lceil 2^{p} \rceil} \leq \frac{l_t(S_{i-1}^+)}{2^{p}}.
\]
In this case, \eqref{eq:Aqproof_eq2} has been satisfied at an earlier step $i'$ with $i'< i$.
We obtain
\begin{align*}
l_t(S^+_i) \leq \frac{l_t(S_{i-1}^+)}{ 2^{p}} \leq \dots \leq \frac{l_t(S_{i'}^+)}{ 2^{p(i-i')}} \leq \biggl( 1+ \frac{1}{\lfloor 2^{p} \rfloor} \biggr) \frac{(1-\gamma) L^p}{ 2^{pi}}
\end{align*}
by using the upper bound for $S_{i'}^+$.
Thus, we have
\[
\frac{(1-\gamma) L^p}{2^{pi}} \leq l_t(S^+_i) \leq \biggl( 1+ \frac{1}{\lfloor 2^{p} \rfloor} \biggr) \frac{(1-\gamma) L^p}{2^{pi}}
\]
for every $S^+_i$.

Let $U^{--}_i = U^-_i - (0,(1+\gamma)L_i^p)$.
We have a collection $\{ S^+_{i,j} \}_{i,j}$ of pairwise disjoint rectangles. 
However, the rectangles in the corresponding collections $\{ U^-_{i,j} \}_{i,j}$ and $\{ U^{--}_{i,j} \}_{i,j}$ may overlap. 
Thus, we replace them by subfamilies $\{ \widetilde{U}^-_{i,j} \}_{i,j}$ and $\{ \widetilde{U}^{--}_{i,j} \}_{i,j}$ of pairwise disjoint rectangles, which are constructed in the following way.
At the first step, choose $\{ U^-_{1,j} \}_{j}$ and $\{ U^{--}_{1,j} \}_{j}$ and denote them by $\{ \widetilde{U}^-_{1,j} \}_j$ and $\{ \widetilde{U}^{--}_{1,j} \}_j$. 
Then consider the collections $\{ U^-_{2,j} \}_{j}$ and $\{ U^{--}_{2,j} \}_{j}$ where each $U^-_{2,j}$ and $U^{--}_{2,j}$ either intersects some $\widetilde{U}^-_{1,j}$ or $\widetilde{U}^{--}_{1,j}$, or does not intersect any $\widetilde{U}^-_{1,j}$ or $\widetilde{U}^{--}_{1,j}$.
Select the pairs of rectangles $U^-_{2,j}$, $ U^{--}_{2,j}$ so that neither $U^-_{2,j}$ nor $U^{--}_{2,j}$ intersects any $\widetilde{U}^-_{1,j}$ or $\widetilde{U}^{--}_{1,j}$, and denote the obtained collections by $\{ \widetilde{U}^-_{2,j} \}_j$ and $\{ \widetilde{U}^{--}_{2,j} \}_j$.
At the $i$th step, choose those pairs $U^-_{i,j}$, $U^{--}_{i,j}$ so that neither $U^-_{i,j}$ nor $U^{--}_{i,j}$  intersects any previously selected $\widetilde{U}^-_{i',j}$ or $\widetilde{U}^{--}_{i',j}$, $i' < i$.
Hence, we obtain collections $\{ \widetilde{U}^-_{i,j} \}_{i,j}$ and $\{ \widetilde{U}^{--}_{i,j} \}_{i,j}$ of pairwise disjoint rectangles.
Observe that for every $U^-_{i,j}$ there exists $\widetilde{U}^-_{i',j}$ with $i' < i$ such that
\begin{equation}
\label{eq:Aqproof_plussubset}
\text{pr}_x(U^-_{i,j}) \subset \text{pr}_x(\widetilde{U}^-_{i',j}) \quad \text{and} \quad \text{pr}_t(U^-_{i,j}) \subset 
\biggl( 2\frac{1+\gamma}{1-\gamma} + 2^{1-p}+1 \biggr)
\text{pr}_t(\widetilde{U}^-_{i',j}) ,
\end{equation}
since
\[
\biggl( 2 \frac{1 + \gamma}{1-\gamma} + 2^{1-p} \biggr) \frac{l_t(\widetilde{U}^-_{i',j})}{2} 
= \frac{(1-\gamma)L^p}{2^{p(i'+1)}} + \frac{(1+\gamma)L^p}{2^{pi'}}  
\geq l_t(U^-_{i,j}) + \frac{(1+\gamma)L^p}{2^{pi'}} .
\]
Here pr$_x$ denotes the projection to $\mathbb R^n$ and pr$_t$ denotes the projection to the time axis.
Let 
\[
R^{\tau}_{i,j} = Q(x_{R_{i,j}} ,L_i) \times (t_{R_{i,j}} +\gamma L_i^p - \tau(1+\gamma)L_i^p , t_{R_{i,j}} +L_i^p).
\]
Note that $S^+_{i,j}$ is spatially contained in $U^-_{i,j}$, that is, $\text{pr}_x S^+_{i,j}\subset \text{pr}_x U^-_{i,j}$.
In the time direction, we have
\begin{equation}
\label{eq:Aqproof_minusplussubset}
\text{pr}_t(S^+_{i,j}) \subset \text{pr}_t(R^\tau_{i,j}) 
\subset \biggl( 2\tau \frac{1 + \gamma}{1-\gamma} +1 \biggr) \text{pr}_t(U^-_{i,j}) ,
\end{equation}
since
\[
\biggl( 2\tau \frac{1 + \gamma}{1-\gamma} + 2 \biggr) \frac{l_t(U^-_{i,j})}{2} = \frac{(1-\gamma)L^p}{2^{pi}} + \frac{\tau(1+\gamma)L^p}{2^{pi}} = l_t(R^{\tau}_{i,j}) .
\]
Therefore, by~\eqref{eq:Aqproof_plussubset} and~\eqref{eq:Aqproof_minusplussubset}, it holds that
\begin{equation}
\label{eq:Aqproof_start}
\Big\lvert \bigcup_{i,j} S^+_{i,j} \Big\rvert \leq c_1 \sum_{i,j} \lvert \widetilde{U}^-_{i,j} \rvert 
\quad\text{with}\quad
c_1 = 
\biggl( 2\frac{1+\gamma}{1-\gamma} + 2^{1-p}+1 \biggr)
\biggl( 2\tau \frac{1 + \gamma}{1-\gamma} +1 \biggr).
\end{equation}

For the rest of the proof and to simplify the notation, let $U^-_i = \widetilde{U}^-_{i,j}$ and $U^-_{i-1} = \widetilde{U}^-_{i-1,j'}$ be fixed, where $U^-_i$ was obtained by subdividing the previous $U^-_{i-1}$ for which $\lvert U^-_{i-1} \rvert / w(U^-_{i-1}) \leq \lambda$. 
Our goal is to apply the parabolic doubling property twice to reach from $U^{--}_i$ to $U_{i-1}^-$.
To this end, we create enough space in time by choosing $\tau$ large enough.
More precisely, let $\tau\geq1$ such that
\begin{align*}
\tau(1+\gamma)L^p &=
\frac{\tau(1+\gamma)L^p}{2^p} + \frac{(1+\gamma)L^p}{2^p} 
+ 2\gamma L^p+ \frac{1}{2}(1-\gamma)L^p \\
&\qquad
+ \frac{1}{2} \frac{(1-\gamma)L^p}{2^p}
+ 2^p 2\gamma L^p+ \frac{1}{2} 2^p (1-\gamma)L^p   \\
&\qquad 
+\frac{1}{2} (1-\gamma) L^p
+(1-\gamma)L^p
-\frac{(1-\gamma)L^p}{2^p} ,
\end{align*}
that is, 
\begin{align*}
\tau &= 
\frac{2^p}{2^p-1} \biggl( 2^{p-1} +2 +\frac{1}{2^{p+1}} + \Bigl(2^{p} -2 +\frac{1}{2^p}\Bigr) \frac{\gamma}{1+\gamma} \biggr) .
\end{align*}
With this choice,
we have enough space in time to apply the parabolic doubling condition to reach from $U^{--}_i$ to $U_{i-1}^-$.
More precisely, there exist two parabolic rectangles $P, V$ such that
$U_{i-1}^-\subset P^-(\gamma)$, $V^-(\gamma) =\frac{1}{2}P^+(\gamma)$ and $\frac{1}{2}V^+(\gamma) = U_i^{--}$.
Applying the parabolic doubling condition~\eqref{eq:pardoubling} twice, we obtain
\begin{align*}
w(U^-_{i-1}) \leq w(P^-(\gamma)) \leq c_d w(V^-(\gamma)) \leq c_d^2 w(U_i^{--}) .
\end{align*}
By \eqref{eq:lowerupperpart_bound} in the proof of Lemma~\ref{lemma:timemove}~$(ii)$,
we have
\begin{align*}
\lvert U_i^{--} \cap \{ \rho w > w_{U_i^-} \} \rvert 
&< \frac{\rho}{w_{U_i^-}} w(U_i^{--}) 
= \rho \frac{w(U_i^{--})}{w(U_i^-)} \lvert U_i^- \rvert\\
&\leq
\rho 
\max\Bigl\{1, 2\frac{\beta}{\alpha}\Bigr\}
\lvert U_i^- \rvert
= \alpha \lvert U_i^- \rvert ,
\end{align*}
where $\rho = \alpha / \max\{1, 2\beta/\alpha \} $.
Then by the assumption (qualitative measure condition) it holds that
\[
w(U_i^{--} \cap \{ \rho w > w_{U_i^-} \}) < \beta w(U_i^-) ,
\]
which implies
\[
w(U_i^{--}) < w(U_i^{--} \cap \{ \rho w \leq w_{U_i^-} \}) + \beta w(U_i^-) . 
\]
Combining the estimates above, we obtain
\begin{align*}
\frac{2^{n+p} }{c_d^2 \lambda} \lvert U_i^- \rvert
&= \frac{1}{c_d^2 \lambda} \lvert U_{i-1}^{-} \rvert
\leq \frac{1}{c_d^2} w(U_{i-1}^-)
\leq w(U_i^{--}) \\
&\leq w(U_i^{--} \cap \{ \rho w \leq w_{U_i^-} \}) + \beta w(U_i^-)
\\
&\leq w(U_i^{--} \cap \{ \rho w \leq w_{U_i^-} \}) + \frac{\beta}{\lambda} \lvert U_i^- \rvert ,
\end{align*}
and thus
\[
\biggl( \frac{2^{n+p}}{c_d^2} - \beta \biggr) \lvert U_i^- \rvert \leq \lambda w(U_i^{--} \cap \{ \rho w \leq w_{U_i^-} \}) .
\]
Since $\beta < 2^{n+p-1}/c_d^2$ and $w_{U^-_{i}} < \lambda^{-1} $, we have
\begin{equation}
\label{eq:Aqproof_sublevelestimate}
\lvert U_i^- \rvert \leq c_2 \lambda w(U_i^{--} \cap \{ \rho w \leq w_{U_i^-} \}) 
\leq
c_2 \lambda w(U_i^{--} \cap \{ w^{-1} > \rho \lambda \}) ,
\end{equation}
where $c_2 = c_d^2/2^{n+p-1} $.

If $(x,t) \in S^+_0 \setminus \bigcup_{i,j} S^+_{i,j}$, then there exists a sequence of subrectangles $S^+_l$ containing $(x,t)$ such that 
\[
\frac{\lvert U^-_l \rvert}{w(U^-_l)} = \dashint_{U^-_l} f \dmu \leq \lambda
\]
and $\lvert S^+_l \rvert \to 0$ as $l \to \infty$.
The Lebesgue differentiation theorem~\cite[Lemma~2.3]{KinnunenMyyryYang2022}
implies that $w^{-1} = f(x,t) \leq \lambda$
for almost every $(x,t) \in S^+_0 \setminus \bigcup_{i,j} S^+_{i,j}$.
It follows that
\[
S^+_0 \cap \{ w^{-1} > \lambda \} \subset \bigcup_{i,j} S^+_{i,j}
\]
up to a set of measure zero.
Using this together with~\eqref{eq:Aqproof_start} and~\eqref{eq:Aqproof_sublevelestimate}, we obtain
\begin{align*}
\lvert S^+_0 \cap \{ w^{-1} > \lambda \} \rvert &\leq 
c_1 \sum_{i,j} \lvert \widetilde{U}^-_{i,j} \rvert
\leq c_1 c_2 \lambda \sum_{i,j} w(\widetilde{U}^{--}_{i,j} \cap \{ w^{-1} > \rho \lambda \}) \\
&\leq c \lambda w(R_0^{\tau} \cap \{ w^{-1} > \rho \lambda \}) ,
\end{align*}
where $c = c_1 c_2 $.
This completes the proof.
\end{proof}

The following theorem shows that 
the parabolic reverse H\"older inequality
together with the parabolic doubling condition
implies the parabolic Muckenhoupt condition.

\begin{theorem}
\label{rhar}
Let $1<q<\infty$
and $w \in RH_q^+$
satisfying~\eqref{eq:pardoubling} with $0<\gamma<1$.
Then $w\in A^+_r(\gamma)$ for some $r>1$.
\end{theorem}

\begin{proof}
By Lemma~\ref{lem:RHItimelag} and the proof of Theorem~\ref{thm:quantiAinfty}, we see that the assumptions of Lemma~\ref{lemma:measureweight-estimate} are satisfied and thus it can be applied.
Let $R \subset \mathbb R^{n+1}$ be a parabolic rectangle.
Let $\varepsilon>0$ to be chosen later.
Denote $B = (w_{U^-})^{-1}$.
We show that Lemma~\ref{lemma:measureweight-estimate} implies the corresponding estimate for the truncated weight $\max\{w,1/k\}$, $ k \in \mathbb N $, that is,
\begin{equation}
\label{eq:RHAq_trunc}
\lvert R^{+}(\gamma) \cap \{ w_k^{-1} > \lambda \} \rvert \leq c \lambda w_k( R^{\tau} \cap \{ w_k^{-1} > \rho \lambda \}  ) .
\end{equation}
If $\lambda\geq k$, then $\{ w_k^{-1} > \lambda \} = \emptyset $ and thus the estimate holds.
On the other hand,
if $\lambda < k$,
then $\{ w_k^{-1} > \lambda \} = \{ w^{-1} > \lambda \}$ and $\{ w_k^{-1} > \rho \lambda \}  = \{ w^{-1} > \rho \lambda \} $.
Hence, \eqref{eq:RHAq_trunc} holds true.

Applying Cavalieri's principle with 
\eqref{eq:RHAq_trunc}, we obtain
\begin{align*}
\int_{R^+(\gamma)} w_k^{-\varepsilon} &= \varepsilon \int_0^\infty \lambda^{\varepsilon-1} \lvert R^{+}(\gamma) \cap \{ w_k^{-1} > \lambda \} \rvert \dla \\
&= \varepsilon \int_0^{B} \lambda^{\varepsilon-1} \lvert R^{+}(\gamma) \cap \{ w_k^{-1} > \lambda \} \rvert \dla\\
&\qquad + \varepsilon \int_{B}^\infty \lambda^{\varepsilon-1} \lvert R^{+}(\gamma) \cap \{ w_k^{-1} > \lambda \} \rvert \dla \\
&\leq \lvert R^{+}(\gamma) \rvert \varepsilon \int_0^{B} \lambda^{\varepsilon-1} \dla + c \varepsilon \int_{B}^\infty \lambda^{\varepsilon} w_k( R^{\tau} \cap \{ w_k^{-1} > \rho \lambda \}  ) \dla \\
&\leq \lvert R^{+}(\gamma) \rvert B^\varepsilon + \frac{c \varepsilon}{\rho^{1+\varepsilon}} \int_0^\infty \lambda^{\varepsilon} w_k( R^{\tau} \cap \{ w_k^{-1} > \lambda \}  ) \dla \\
&\leq \lvert U^- \rvert (w_{U^-})^{-\varepsilon} + \frac{c}{\rho^{1+\varepsilon}} \frac{\varepsilon}{1+\varepsilon} \int_{R^{\tau}} w_k^{-\varepsilon} .
\end{align*}
By choosing $\varepsilon>0$ to be small enough, we can absorb the integral over $R^+(\gamma)$ of the second term to the left-hand side to get
\begin{align*}
\biggl( 1- \frac{c}{\rho^{1+\varepsilon}} \frac{\varepsilon}{1+\varepsilon} \biggr) \int_{R^+(\gamma)} w_k^{-\varepsilon} \leq \lvert U^- \rvert (w_{U^-})^{-\varepsilon} + \frac{c}{\rho^{1+\varepsilon}} \frac{\varepsilon}{1+\varepsilon} \int_{R^{\tau}\setminus R^+(\gamma)} w_k^{-\varepsilon} .
\end{align*}
Denote $R^{\tau,-} = R^{\tau} \setminus R^+(\gamma)$.
Hence, we have
\begin{equation}
\label{eq:qualiproof_cavalieri_iteration}
\int_{R^+(\gamma)} w_k^{-\varepsilon} \leq c_0 \lvert U^- \rvert (w_{U^-})^{-\varepsilon} + c_1 \varepsilon \int_{R^{\tau,-}} w_k^{-\varepsilon} ,
\end{equation}
where
\[
c_0 = \frac{1+\varepsilon }{1-(c \rho^{-1-\varepsilon} -1) \varepsilon} \quad \text{and} \quad c_1 = \frac{c \rho^{-1-\varepsilon} }{1-(c \rho^{-1-\varepsilon} -1) \varepsilon} .
\]

Fix $R_0 = Q(x_0,L) \times (t_0 - L^p, t_0 + L^p) \subset \mathbb R^{n+1}$.
We cover $R^{\tau,-}_{0}(\gamma)$ by 
\[
M = 2^n \biggl\lceil \frac{\tau(1+\gamma)}{(1-\gamma)/2^p} \biggr\rceil = 2^{n} \biggl\lceil 2^p \tau \frac{1+\gamma}{1-\gamma} \biggr\rceil
\]
rectangles $R^+_{1,j}(\gamma)$ with spatial side length $ L_1 = L/2$ and time length $ (1-\gamma)L_1^p = (1-\gamma) L^p / 2^p$. This can be done by dividing each spatial edge of $R^{\tau,-}_{0}(\gamma)$ into two equally long pairwise disjoint intervals, and the time interval of $R^{\tau,-}_{0}(\gamma)$ into $\lceil 2^p \tau (1+\gamma) /(1-\gamma) \rceil$ equally long intervals such that their overlap is bounded by $2$.
Thus, the overlap of $R^+_{1,j}(\gamma)$ is bounded by $2$.
Then consider $R^{\tau,-}_{1,j}(\gamma)$ and cover it in the same way as before by $M$ rectangles $R^+_{2,j}(\gamma)$ with spatial side length $ L_2 = L/2^{2}$ and time length $ (1-\gamma)L_2^p = (1-\gamma)L^p / 2^{2p}$.
At the $i$th step, cover $R^{\tau,-}_{i-1,j}(\gamma)$ by $M$ rectangles $R^+_{i,j}(\gamma)$ with spatial side length $ L_i = L/2^{i}$ and time length $ (1-\gamma)L_i^p = (1-\gamma)L^p / 2^{pi}$ such that their overlap is bounded by $2$ for fixed $R^{\tau,-}_{i-1,j}(\gamma)$.
We observe that  the bottom of $R_0^+(\gamma)$ is time distance at most
\begin{equation}
\label{eq:qualiproof_maxdist-lowerparts}
\sum_{i=0}^\infty l_t(R^{\tau,-}_{i,j}(\gamma)) = \sum_{i=0}^\infty \frac{\tau(1+\gamma)L^p}{2^{pi}} = \frac{2^p}{2^p -1} \tau(1+\gamma)L^p 
\end{equation}
above the bottom of $U^-_{i,j}$.

We construct a chain 
of rectangles
from each $U^-_{i,j}$ to 
$U^{\sigma,-}_{0} = R^+(\gamma) - (0, \sigma (1+\gamma) L^p )$, where $\sigma\geq\tau$ is chosen later.
Fix $U^-_{i} = U^-_{i,j}$.
Let $N=i$ denote the number of rectangles in the chain and $d_{i,m}$, $m\in \{1,\dots,N\}$, the distances between the bottoms of the rectangles 
given by
\begin{align*}
d_{i,m} &= 2^{mp} (1+\gamma) L_i^p + \frac{1}{2} (2^{mp}-2^{(m-1)p}) (1-\gamma)L_i^p \\
&\qquad + 2^{(m+1)p} (1+\gamma) L_i^p + \frac{1}{2} (2^{(m+1)p}-2^{mp}) (1-\gamma)L_i^p \\
&= 
2^{mp} (2^p+1) (1+\gamma) L_i^p + 2^{mp-1} (2^{p}-2^{-p}) (1-\gamma)L_i^p .
\end{align*}
Define the elements of the chain by
\begin{align*}
V_0 = U^-_i = Q(x_{R_{i}},L_i) \times (a_0 , a_0 + (1-\gamma) L_i^p) \quad\text{and}\quad V_m = Q_{m} \times I_m
\end{align*}
for every $m\in \{1,\dots,N\}$, where
\begin{align*}
Q_m &= 2^m Q(x_{R_{i}},L_i) + \frac{2^{m}-1}{2^{i}-1}(x_{R_{0}} - x_{R_{i}}) , \\
I_m &= (a_m,b_m) = (a_{m-1} - d_{i,m} ,  a_{m-1} + 2^{mp} (1-\gamma) L_i^p - d_{i,m}  ) .
\end{align*}
Observe that $Q_0 = pr_x(U^-_{i})$, $Q_{N} = pr_x(U^-_{0,\sigma})$ and 
$\lvert V_{m} \vert = 2^{n+p}\lvert V_{m-1} \rvert$.
The bottom of $V_0$ is time distance 
\begin{align*}
\sum_{m=1}^{N} d_{i,m} &=
\sum_{m=1}^{i} 2^{mp} (2^p+1) (1+\gamma) L_i^p + 2^{mp-1} (2^{p}-2^{-p}) (1-\gamma)L_i^p \\
& = 
\frac{2^{2p}+2^p}{2^p -1} \frac{2^{pi}-1}{2^{pi}} (1+\gamma) L^p
+ \frac{2^{2p}-1 }{2^{p+1} -2} \frac{2^{pi}-1}{2^{pi}} (1-\gamma) L^p 
\end{align*}
above the bottom of $V_N$.
Hence, the bottom of $V_0$ is time distance at most
\begin{equation}
\label{eq:qualiproof_maxdist-chain}
\sum_{m=1}^{\infty} d_{i,m} =
\frac{2^{2p}+2^p}{2^p -1} (1+\gamma) L^p
+ \frac{2^{2p}-1 }{2^{p+1} -2} (1-\gamma) L^p 
\end{equation}
above the bottom of $V_N$.
By combining~\eqref{eq:qualiproof_maxdist-lowerparts} and~\eqref{eq:qualiproof_maxdist-chain}, 
we obtain an upper bound for the time length
from the bottom of $R_0^+(\gamma)$ to the bottom of $V_N$. Based on this, we fix $U^{\sigma,-}_{0}$ by choosing $\sigma$ such that
\begin{align*}
\sigma (1+\gamma)L^p 
&=\sum_{i=0}^\infty l_t(R^{\tau,-}_{i,j}(\gamma))+\sum_{m=1}^{\infty} d_{i,m} \\
&= \frac{2^p}{2^p -1} \tau(1+\gamma)L^p +  \frac{2^{2p}+2^p}{2^p -1} (1+\gamma) L^p
+ \frac{2^{2p}-1 }{2^{p+1} -2} (1-\gamma) L^p ,
\end{align*}
that is,
\[
\sigma = \frac{2^p\tau}{2^p -1} + \frac{2^{2p}+2^p}{2^p -1}
+ \frac{2^{2p}-1 }{2^{p+1} -2} \frac{1-\gamma}{1+\gamma} .
\]

We add one more rectangle $V_{N+1}$ into the chain so that the chain would end at $U^{\sigma,-}_{0}$.
Let
\[
V_{N+1} = V_N - (0, b_i l_t(V_N) ) = V_N - (0, b_i 2^{pi} (1-\gamma) L_{i}^p ) = 
V_N - ( 0, b_i (1-\gamma) L^p ) , 
\]
where $b_i$ is chosen such that the bottom of $V_{N+1}$ intersects with the bottom of $U^{\sigma,-}_{0}$.
Then $U^{\sigma,-}_{0}$ is contained in $V_{N+1}$.
Next we find an upper bound for $b_i$.
We observe that a rough lower bound for
the time length from the bottom of $R^+_0(\gamma)$ to the bottom $V_N$ is
given by
\[
\frac{(1-\gamma)L^p}{2^p} + 
(2^p+1) (1+\gamma) L^p + 2^{-1} (2^{p}-2^{-p}) (1-\gamma)L^p .
\]
Therefore, the bottom of $V_N$ is time distance at most
\begin{align*}
& \sigma (1+\gamma)L^p - 
(2^p+1) (1+\gamma) L^p - 2^{-1} (2^{p}+2^{-p}) (1-\gamma)L^p
\\
&\qquad =
\frac{2^p}{2^p -1} \tau(1+\gamma)L^p +  \frac{2^p+1}{2^p -1} (1+\gamma) L^p
- \frac{2^{p}+2^{-p} }{2^{p+1} -2} (1-\gamma) L^p 
\end{align*}
above the bottom of $U^{\sigma,-}_{0}$.
By this,
we obtain an upper bound for $b_i$
\[
b_i (1-\gamma) L^p 
\leq \frac{2^p}{2^p -1} \tau(1+\gamma)L^p +  \frac{2^p+1}{2^p -1} (1+\gamma) L^p
- \frac{2^{p}+2^{-p} }{2^{p+1} -2} (1-\gamma) L^p ,
\]
that is,
\[
b_i \leq \frac{2^p \tau + 2^p + 1}{2^p -1} \frac{1+\gamma}{1-\gamma} - \frac{2^{p}+2^{-p} }{2^{p+1} -2}  = \theta .
\]

By the definition of $V_m$,  
we can apply the parabolic doubling condition~\eqref{eq:pardoubling} twice for each pair of $V_{m-1}, V_m$, $m\in\{1,\dots,N\}$,
and Lemma~\ref{lemma:timemove}~$(ii)$ for $V_{N}, V_{N+1}$ with $\theta\geq\sup b_i$ to get
\[
w(V_0) \geq c_d^{-2} w(V_1) \geq c_d^{-2N} w(V_N) \geq 
c_d^{-2i}
\frac{1}{c_2}
w(V_{N+1}) ,
\]
where $c_2$
is the constant from Lemma~\ref{lemma:timemove}~$(ii)$.
We conclude that
\begin{equation}
\label{eq:qualiproof_chainestimate2}
w(U^{\sigma,-}_{0}) \leq w(V_{N+1}) 
\leq 
c_2 c_d^{2i} w(V_0) 
\leq c_2 \xi^i w(U^-_i) ,
\end{equation}
where 
$\xi = c_d^{2}$.

We iterate~\eqref{eq:qualiproof_cavalieri_iteration} to obtain
\begin{align*}
\int_{R^+_0(\gamma)} w_k^{-\varepsilon} &\leq c_0 \lvert U^-_0 \rvert (w_{U^-_0})^{-\varepsilon} + c_1 \varepsilon \int_{R^{\tau,-}_{0}} w_k^{-\varepsilon} \\
&\leq c_0 \lvert U^-_0 \rvert (w_{U^-_0})^{-\varepsilon} + c_1 \varepsilon \sum_{j=1}^M \int_{R^+_{1,j}(\gamma)} w_k^{-\varepsilon} \\
&\leq c_0 \lvert U^-_0 \rvert (w_{U^-_0})^{-\varepsilon} \\
&\qquad+ c_1 \varepsilon \sum_{j=1}^M \biggl( c_0 \lvert U^-_{1,j} \rvert (w_{U^-_{1,j}})^{-\varepsilon} + c_1 \varepsilon \int_{R^{\tau,-}_{1,j}(\gamma)} w_k^{-\varepsilon} \biggr) \\
&= c_0 \lvert U^-_0 \rvert (w_{U^-_0})^{-\varepsilon} \\
&\qquad+ c_0  c_1 \varepsilon \sum_{j=1}^M \lvert U^-_{1,j} \rvert (w_{U^-_{1,j}})^{-\varepsilon} + (c_1 \varepsilon)^2 \sum_{j=1}^M \int_{R^{\tau,-}_{1,j}(\gamma)} w_k^{-\varepsilon} \\
&\leq c_0 \sum_{i=0}^N \biggl( (c_1 \varepsilon)^{i} \sum_{j=1}^{M^i} \lvert U^-_{i,j} \rvert (w_{U^-_{i,j}})^{-\varepsilon} \biggr) + (c_1 \varepsilon)^{N+1} \sum_{j=1}^{M^N} \int_{R^{\tau,-}_{N,j}(\gamma)} w_k^{-\varepsilon} \\
&\leq c_0 \sum_{i=0}^N \biggl( (c_1 \varepsilon)^{i} \sum_{j=1}^{M^i} \lvert U^-_{i,j} \rvert (w_{U^-_{i,j}})^{-\varepsilon} \biggr) + (c_1 \varepsilon)^{N+1} M^N \int_{R^{\sigma,-}_{0}(\gamma)} w_k^{-\varepsilon} \\
&= I + II .
\end{align*}
We observe that $II$ tends to zero if $\varepsilon < \frac{1}{c_1 M}$ as $N \to \infty$ since
$w_k^{-\varepsilon}$ is bounded.
For the inner sum of the first term $I$, we apply~\eqref{eq:qualiproof_chainestimate2} to get
\begin{align*}
\sum_{j=1}^{M^i} \lvert U^-_{i,j} \rvert (w_{U^-_{i,j}})^{-\varepsilon} 
&= \sum_{j=1}^{M^i} \lvert U^-_{i,j} \rvert^{1+\varepsilon} w(U^-_{i,j})^{-\varepsilon}\\
& \leq \sum_{j=1}^{M^i} 2^{-(n+p)(1+\varepsilon) i} \lvert U^{\sigma,-}_{0} \rvert^{1+\varepsilon} w(U^-_{i,j})^{-\varepsilon} \\
&\leq \sum_{j=1}^{M^i} 2^{-(n+p)(1+\varepsilon) i} \lvert U^{\sigma,-}_{0} \rvert^{1+\varepsilon} c_2^\varepsilon \xi^{\varepsilon i} w(U^{\sigma,-}_{0})^{-\varepsilon} \\
&= 2^{-(n+p)(1+\varepsilon) i}  c_2^\varepsilon \xi^{\varepsilon i} M^i \lvert U^{\sigma,-}_{0} \rvert (w_{U^{\sigma,-}_{0}})^{-\varepsilon} .
\end{align*}
Thus, it follows that
\begin{align*}
I &\leq c_0 \sum_{i=0}^N (c_1 \varepsilon)^{i} 2^{-(n+p)(1+\varepsilon) i} c_2^\varepsilon \xi^{\varepsilon i} M^i \lvert U^{\sigma,-}_{0} \rvert (w_{U^{\sigma,-}_{0}})^{-\varepsilon} \\
&\leq c_0 c_2^\varepsilon \lvert U^{\sigma,-}_{0} \rvert (w_{U^{\sigma,-}_{0}})^{-\varepsilon} \sum_{i=0}^N (c_1 \varepsilon)^{i} 2^{-(n+p)(1+\varepsilon) i}  \xi^{\varepsilon i} M^i .
\end{align*}
We estimate the sum by
\begin{align*}
\sum_{i=0}^N (c_1 \varepsilon)^{i} 2^{-(n+p)(1+\varepsilon) i}  \xi^{\varepsilon i} M^i &= \sum_{i=0}^N \bigl( c_1 \varepsilon 2^{-(n+p)(1+\varepsilon)}  \xi^{\varepsilon} M \bigr)^i \\
&\leq \frac{1}{1- c_1 \varepsilon 2^{-(n+p)(1+\varepsilon)}  \xi^{\varepsilon} M} ,
\end{align*}
whenever $\varepsilon$ is small enough, for example, $\varepsilon < 2^{n+p} /(c_1 \xi M)$.
Then it holds that
\begin{align*}
\int_{R^+_0(\gamma)} w_k^{-\varepsilon}
&\leq \frac{c_0 c_2^\varepsilon}{1- c_1 \varepsilon 2^{-(n+p)(1+\varepsilon)}  \xi^{\varepsilon} M} \lvert U^{\sigma,-}_{0} \rvert (w_{U^{\sigma,-}_{0}})^{-\varepsilon}
\end{align*}
for small enough $\varepsilon$.
By applying Fatou's lemma as $k \to \infty$, we obtain
\[
\dashint_{U^{\sigma,-}_{0}} w \biggl( \dashint_{R^+_0(\gamma)} w^{-\varepsilon} \biggr)^\frac{1}{\varepsilon} \leq c_3 ,
\]
where
\[
c_3^\varepsilon = \frac{c_0 c_2^\varepsilon}{1- c_1 \varepsilon 2^{-(n+p)(1+\varepsilon)}  \xi^{\varepsilon} M} .
\]
By
\cite[Theorem~3.1]{KinnunenMyyry2023},
we conclude that $w \in A^+_r(\gamma)$ with $r = 1+1/\varepsilon$.
This completes the proof.
\end{proof}

\end{document}